\documentclass{amsart}

\usepackage[T1]{fontenc}
\usepackage[utf8]{inputenc}
\usepackage{lmodern}
\usepackage{amsmath,amssymb,amscd,mathtools}
\usepackage{graphicx}
\usepackage{float}
\usepackage{tikz}
\usepackage[all]{xy}
\usepackage{hyperref}

\numberwithin{equation}{section}

\newtheorem{theorem}{Theorem}[section]
\newtheorem{proposition}[theorem]{Proposition}
\newtheorem{lemma}[theorem]{Lemma}

\theoremstyle{remark}
\newtheorem{remark}[theorem]{Remark}
\theoremstyle{plain}

\newcommand{\C}{\mathcal C}
\newcommand{\A}{\mathcal A}
\newcommand{\B}{\mathcal B}
\newcommand{\Z}{\mathcal Z}
\newcommand{\D}{\mathcal D}

\newcommand{\M}{\mathcal M}
\newcommand{\Irr}{\mathrm{Irr}}
\newcommand{\FPdim}{\mathrm{FPdim}}
\newcommand{\Hom}{\mathrm{Hom}}
\renewcommand{\Vec}{\mathrm{Vec}}
\newcommand{\FP}{\mathrm{FP}}
\newcommand{\Dim}{\operatorname{Dim}}
\newcommand{\EVAL}{\mathrm{EVAL}}
\newcommand{\TVEVAL}{\operatorname{TV\text{-}EVAL}}
\newcommand{\RTEVAL}{\operatorname{RT\text{-}EVAL}}
\newcommand{\ConnSimp}{\mathrm{conn,simp}}
\newcommand{\defnfont}[1]{\emph{#1}}

\title{A Complexity Dichotomy for Quantum Invariants of 3-Manifolds}
\author{César Galindo}
\address{Departamento de Matemáticas, Universidad de los Andes, Bogotá, D.C.
111711, Colombia}
\email{cn.galindo1116@uniandes.edu.co}

\date{\today}

\begin{document}

\begin{abstract}
We determine the complexity of exact evaluation of the Reshetikhin--Turaev and
Turaev--Viro invariants of closed connected oriented 3-manifolds, with the
underlying tensor category fixed. If \(\C\) is a modular category, then the
Reshetikhin--Turaev invariant \(Z_\C(M)\) can be computed in polynomial time from
a framed-link surgery presentation of \(M\) precisely when \(\C\) is pointed;
otherwise the problem is \(\#\mathrm P\)-hard. If \(\A\) is a spherical fusion
category, then the Turaev--Viro invariant \(|M|_\A\) can be computed in
polynomial time from a triangulation of \(M\) precisely when the Drinfeld center
\(\Z(\A)\) is pointed, equivalently when \(\A\) is trivializable pointed;
otherwise the problem is \(\#\mathrm P\)-hard. This proves the dichotomy
conjectured by Bridges and Samperton and identifies the categorical obstruction
to polynomial-time evaluation.
\end{abstract}

\maketitle
\section{Introduction}
\label{sec:introduction}

Quantum invariants of 3-manifolds provide a meeting point for low-dimensional
topology, tensor categories, statistical mechanics, and quantum computation.
Among the most prominent examples are the Reshetikhin--Turaev and
Turaev--Viro invariants, defined respectively from modular categories and from
spherical fusion categories
\cite{ReshetikhinTuraev1991,TuraevViro1992,BarrettWestbury1996,TuraevBook,
TuraevVirelizier2017,BakalovKirillov2001,EGNO2015}.  Both invariants admit
finite combinatorial descriptions: surgery presentations in the
Reshetikhin--Turaev case and triangulations in the Turaev--Viro case.  A finite
formula, however, does not by itself settle the complexity question.  For
a fixed category, how hard is the resulting invariant as the 3-manifold varies?

We work throughout in the exact-evaluation model.  The categorical data are
fixed, together with a number field containing the structural constants and the
normalization scalars; the input is only a combinatorial presentation of the
3-manifold.  Thus \(\FP\) means polynomial-time computation in this fixed
number field, and \(\#\mathrm P\)-hardness is understood under polynomial-time
Turing reductions from integer-valued \(\#\mathrm P\) counting problems
\cite{Papadimitriou1994}.

For a modular category \(\C\), let \(\RTEVAL(\C)\) denote the problem of
computing \(Z_\C(M)\) from a framed-link surgery presentation.  For a spherical
fusion category \(\A\), let \(\TVEVAL(\A)\) denote the problem of computing
\(|M|_\A\) from a triangulation.  The main result identifies the dividing line
in terms of the fixed category: \(\RTEVAL(\C)\) is polynomial-time computable
precisely when \(\C\) is pointed, while \(\TVEVAL(\A)\) is polynomial-time
computable precisely when \(\Z(\A)\) is pointed, equivalently when \(\A\) is
trivializable pointed.  In all other cases the corresponding problem is
\(\#\mathrm P\)-hard.

Complexity questions entered quantum topology early through evaluations of
link polynomials.  Jaeger, Vertigan, and Welsh studied the complexity of
evaluating the Jones and Tutte polynomials \cite{JaegerVertiganWelsh1990}, while
approximation of quantum link invariants became closely connected with quantum
computation
\cite{FreedmanKitaevWang2002,FreedmanLarsenWang2002,AharonovJonesLandau2009,
KuperbergJonesApprox2015}.  Further connections between complexity assumptions
and Jones-polynomial evaluations in topology were studied in
\cite{CuiFreedmanWang2016}.  From the categorical point of view, related
questions appear in the classification of anyonic systems and in finiteness and
universality properties of braid group representations
\cite{NaiduRowell2011,RowellWang2018}.

In the 3-manifold setting, Kirby and Melvin gave an early hardness result.  They
showed that the \(SU(2)\)
Witten--Reshetikhin--Turaev invariant at the fourth root of unity, equivalently the
level-two or Ising theory, determines the number of zeros of cubic forms over
\(\mathbb Z_2\), and concluded that the corresponding evaluation problem
is NP-hard \cite{KirbyMelvinLocalSurgery2004}. Thus the Ising theory already
exhibits hardness phenomena for 3-manifold evaluation. Related formulas
at exceptional roots include the sixth-root formula of Kirby--Melvin--Zhang and
the metaplectic link invariants of Goldschmidt--Jones
\cite{KirbyMelvinZhang1993,GoldschmidtJones1989}; for computational aspects of
metaplectic modular categories, see \cite{HastingsNayakWang2014}.  Further
complexity phenomena for quantum and finite-group 3-manifold invariants were
studied by Kuperberg and Samperton
\cite{KuperbergSampertonZombies2018,KuperbergSampertonColoring2021}.  On the
algorithmic side, Turaev--Viro type invariants have also been studied from the
viewpoint of effective computation and parameterized complexity
\cite{BurtonMariaSpreer2018,MariaSpreer2016,MariaSpreer2020}.  More recently,
Delaney, Maria, and Samperton studied Tambara--Yamagami quantum invariants,
proving hardness results and fixed-parameter algorithms in terms of the first Betti
number \cite{DelaneyMariaSampertonTY2025}.

This paper is motivated by the conjecture of Bridges and
Samperton \cite{BridgesSampertonTQFTDichotomy2025}. They prove a
polynomial-time/\(\#\mathrm P\)-hard dichotomy for tensor contraction problems
associated with Turaev--Viro--Barrett--Westbury and Reshetikhin--Turaev type
TQFTs, using the Cai--Chen dichotomy theorem for weighted constraint
satisfaction problems over \(\mathbb C\) \cite{CaiChen2017}. They also formulate the analogous
dichotomy for closed 3-manifold invariants as a conjecture, leaving open the categorical property
that separates the polynomial-time cases from the \(\#\mathrm P\)-hard ones.
The results of this paper
prove that conjecture and identify the
separating property intrinsically.

We first state the Reshetikhin--Turaev dichotomy.

\begin{theorem}
Let \(\mathcal C\) be a modular category over an algebraically closed field of
characteristic zero, with a fixed number field containing the relevant scalars.
If \(\mathcal C\) is pointed, then
\[
        \RTEVAL(\mathcal C) \in \FP.
\]
If \(\mathcal C\) is not pointed, then
\(\RTEVAL(\mathcal C)\) is \(\#\mathrm P\)-hard.
\end{theorem}

Thus the Reshetikhin--Turaev evaluation problem satisfies a dichotomy governed
by whether the modular category is pointed.  In the pointed case, the simple
objects form a finite abelian group, and the invariant reduces to a Gauss sum
over this group, computable in polynomial time by finite abelian linear algebra;
see Section~\ref{sec:pointed-modular-categories-v2}. In the hard direction,
every non-pointed modular category yields a \(\#\mathrm P\)-hard graph-counting
problem; see Sections~\ref{sec:anomaly-free-case-v2} and
\ref{sec:arbitrary-modular-v2}.

We next state the Turaev--Viro dichotomy.  For every spherical fusion category
\(\mathcal A\), the polynomial-time case is exactly the case in which
\(\mathcal Z(\mathcal A)\) is pointed.  By
\cite[Theorem~3.2]{AngionoGalindo2017}, this is equivalent to saying that
\(\mathcal A\) is tensor equivalent to \(\mathrm{Vec}^{\omega}_{\Lambda}\) for
a finite abelian group \(\Lambda\), with \([\omega]\) satisfying the
trivializability condition recalled in Section~\ref{sec:tv-dichotomy-v2}.

\begin{theorem}
Let \(\mathcal A\) be a spherical fusion category over an algebraically closed field
of characteristic zero, with a fixed number field containing the relevant scalars.
If \(\mathcal A\) is trivializable pointed, then
\[
        \TVEVAL(\mathcal A) \in \FP.
\]
If \(\mathcal A\) is not trivializable pointed, then
\(\TVEVAL(\mathcal A)\) is \(\#\mathrm P\)-hard.
\end{theorem}

The geometric input is an explicit construction assigning to each finite graph
\(G\) a graph manifold \(M_G\). For a graph \(G\), write \(V(G)\) and
\(E(G)\) for its vertex and edge sets.  The construction assigns to each vertex
of \(G\) the block
\[
        \Sigma_{1,r}\times S^1,
\]
where \(r\) is the valence of the vertex, and glues the boundary tori according
to the edges of \(G\) by a fixed mapping class.  We later use the corresponding
four-dimensional plumbing description: \(M_G\) is the boundary of a plumbing
with one torus vertex for each vertex of \(G\); see
Proposition~\ref{prop:MG-plumbing-model-v2}. The construction is effective:
from \(G\) one obtains, in polynomial time,
a framed-link surgery presentation of \(M_G\), and from it a triangulation of
\(M_G\). These presentations have polynomial size in the size of \(G\). In
the anomaly-free Reshetikhin--Turaev case the \(\#\mathrm P\)-hardness is already
witnessed by the manifolds \(M_G\) themselves. For arbitrary modular categories,
the center reduction uses the same family together with its orientation reverses.
For Turaev--Viro invariants it is witnessed by triangulations produced from the
same family.

The computation needed for the reduction is a Reshetikhin--Turaev formula for
these graph manifolds.
Let \(\mathcal C\) be anomaly-free and modular, with the conventions recalled
in Subsection~\ref{subsec:modular-tensor-categories-v2}, and write
\[
        I=\mathrm{Irr}(\mathcal C), \qquad S=(S_{ij})_{i,j\in I}, \qquad d_i=S_{0i}.
\]
Define the matrix
\[
        A_{\mathcal C}(i,j)=\frac{S_{ij}}{d_i d_j}.
\]
Then Theorem~\ref{thm:graph-manifold-rt-formula-v2} gives
\begin{equation}
\label{eq:intro-rt-graph-partition-identity-v2}
        Z_{\mathcal C}(M_G)
        =
        \D_{\mathcal C}^{\,|E(G)|}
        \sum_{\phi:V(G)\to I}
        \prod_{\{u,v\}\in E(G)}
        A_{\mathcal C}(\phi(u),\phi(v)),
\end{equation}
where \(\D_{\mathcal C}=\Delta_+=\Delta_-\).  Hence, up to a known nonzero
factor depending only on \(\mathcal C\) and \(|E(G)|\), the invariant of \(M_G\)
is the weighted graph homomorphism partition function with weight matrix
\(A_{\mathcal C}\).

Formula~\eqref{eq:intro-rt-graph-partition-identity-v2} is closely related to
Turaev's shadow formula for graph manifolds
\cite[Chapter X, Theorem~9.3.1]{TuraevBook}. Turaev's formula treats a more
general class of graph manifolds, by realizing them as boundaries of
four-dimensional plumbings and evaluating the corresponding shadow state sums.
However, its categorical hypotheses are not the ones needed here: Turaev's
formula assumes that every self-dual simple object has second
Frobenius--Schur indicator \(+1\). Our reduction instead has to work in the
anomaly-free setting, because it passes to Drinfeld centers; those centers are
anomaly-free, but they need not satisfy the Frobenius--Schur condition. We
therefore give a direct proof of the special graph-manifold formula needed for
the reduction.

The complexity-theoretic input is Cai--Govorov's dichotomy theorem for complex
weighted graph homomorphism partition functions \cite{CaiGovorov2020}.  The
graph-counting problem appearing in
\eqref{eq:intro-rt-graph-partition-identity-v2} is the following: for a fixed
symmetric matrix \(A=(A_{ij})\), compute
\[
        Z_A(G)=
        \sum_{\phi:V(G)\to I}
        \prod_{\{u,v\}\in E(G)}A_{\phi(u),\phi(v)}
\]
on connected simple graphs of bounded degree. The hard side of Cai--Govorov's
dichotomy says that if \(A\) is not multiplicative-block-rank-one, then this
evaluation problem is \(\#\mathrm P\)-hard. We prove that, for the matrix
\(A_{\mathcal C}\), this failure occurs exactly when \(\mathcal C\) is
non-pointed. Hence every non-pointed
anomaly-free modular category gives a \(\#\mathrm P\)-hard
Reshetikhin--Turaev evaluation problem on the family \(M_G\).

The anomaly-free assumption is removed by passing to the Drinfeld center. If
\(\mathcal C\) is modular, then
\[
        \mathcal Z(\mathcal C)\simeq \mathcal C\boxtimes \mathcal C^{\mathrm{rev}},
\]
and \(\mathcal Z(\mathcal C)\) is anomaly-free.  Moreover, if \(\mathcal C\) is
non-pointed, then \(\mathcal Z(\mathcal C)\) is non-pointed.  The result in the
anomaly-free case applied to \(\mathcal Z(\mathcal C)\) therefore shows that
\(\RTEVAL(\mathcal Z(\mathcal C))\) is \(\#\mathrm P\)-hard.  Finally,
multiplicativity under Deligne products and the behavior under reversing the
braiding give
\[
        Z_{\mathcal Z(\mathcal C)}(M)
        =
        Z_{\mathcal C}(M)\,Z_{\mathcal C}(-M),
\]
and hence \(\RTEVAL(\mathcal Z(\mathcal C))\) reduces in polynomial
time to \(\RTEVAL(\mathcal C)\).  This proves the result
for arbitrary non-pointed modular categories.  The Turaev--Viro dichotomy then
follows from the identity
\[
        |M|_{\mathcal A}=Z_{\mathcal Z(\mathcal A)}(M)
\]
and from the criterion for \(\mathcal Z(\mathcal A)\) to be pointed.
In the trivializable pointed case, the pointed state sum is reduced to a finite
abelian Gauss sum.

Thus, in the RT and TV settings considered here there is no intermediate
complexity case: each evaluation problem is either computable by finite abelian
linear algebra and Gauss sums or is \(\#\mathrm P\)-hard. The dividing line is
whether \(\C\) is pointed in the RT case and whether \(\Z(\A)\) is pointed in the
TV case.

The paper is organized as follows.
Section~\ref{sec:modular-categories-quantum-invariants} fixes the categorical
conventions and the evaluation problems for Reshetikhin--Turaev and
Turaev--Viro invariants.
Section~\ref{sec:graph-partition-complexity} recalls the graph partition functions
and the Cai--Govorov dichotomy used in the reductions.
Section~\ref{sec:graph-manifold-family-v2} constructs the graph-manifold family
\(M_G\) and records the effective surgery and triangulation presentations.
Section~\ref{sec:rt-formula-v2} proves the Reshetikhin--Turaev formula for
\(M_G\).
Section~\ref{sec:anomaly-free-case-v2} proves the non-pointed anomaly-free
case.
Section~\ref{sec:pointed-modular-categories-v2} proves polynomial-time computability for
pointed modular categories.
Section~\ref{sec:arbitrary-modular-v2} removes the anomaly-free assumption by
passing to centers.
Finally, Section~\ref{sec:tv-dichotomy-v2} proves the Turaev--Viro dichotomy,
including the polynomial-time algorithm for the trivializable pointed case.

{\bf Acknowledgements.} The author thanks Diego Romero, Eric Samperton, and
Zhenghan Wang for useful comments.

ChatGPT 5.5 was used for language editing, to improve the clarity of portions of
the final manuscript, to assist with bibliographic searches related to specific
points, and to help convert author-created hand-drawn sketches into the final
figures. Apart from these uses, the mathematical content of the paper was
developed by the author.

The author was partially supported by Grant INV-2025-213-3452 from the School of
Science of Universidad de los Andes.

\section{Modular categories and quantum invariants}
\label{sec:modular-categories-quantum-invariants}

Throughout the paper \(k\) is an algebraically closed field of characteristic
zero, and all tensor categories are \(k\)-linear. This section fixes the
terminology and normalizations used in the reductions: spherical fusion
categories for Turaev--Viro invariants, ribbon and modular categories for
Reshetikhin--Turaev invariants, and the corresponding state-sum normalizations.
We use
\cite{ENO2005,EGNO2015} for fusion categories and
\cite{TuraevBook,BakalovKirillov2001,TuraevVirelizier2017} for ribbon, modular,
and state-sum conventions.

\subsection{Fusion, spherical, and ribbon categories}
\label{subsec:fusion-and-modular-categories}

A \defnfont{fusion category} is a semisimple rigid monoidal category with simple
unit, finite-dimensional morphism spaces, and finitely many isomorphism classes
of simple objects. If \(\A\) is fusion, we write \(\Irr(\A)\) for these classes
and choose representatives \(V_i\), \(i\in\Irr(\A)\). The unit class is denoted
by \(0\), and \(i^*\) is determined by \(V_{i^*}\cong V_i^*\); see
\cite[Definition~4.1.1]{EGNO2015} and
\cite[Section~4.5.1]{TuraevVirelizier2017}.
The fusion coefficients are defined by
\[
V_i\otimes V_j\cong \bigoplus_{\ell\in\Irr(\A)}
N_{ij}^{\ell}V_\ell.
\]
For fixed \(i\), the \defnfont{fusion matrix} \(N_i\) is the matrix of
multiplication by \([V_i]\) in the Grothendieck ring, in the basis of simple
objects. Its Perron--Frobenius eigenvalue is the
\defnfont{Frobenius--Perron dimension} \(\FPdim(V_i)\); see
\cite[Chapter~3]{EGNO2015}.

A fusion category is \defnfont{pointed} if every simple object is invertible;
for pointed fusion categories, see
\cite[Definition~5.11.1]{EGNO2015}.

A \defnfont{pivotal structure} on a fusion category is a tensor natural
isomorphism from the identity functor to the double-dual functor. Such a
structure defines left and right pivotal traces of endomorphisms, and it is
\defnfont{spherical} if these traces agree. In a spherical fusion category the
categorical dimension of an object is the pivotal trace of its identity; for a
simple object \(V_i\), we denote it by \(\dim_\A(V_i)\). The global dimension is
\begin{equation}
\label{eq:spherical-global-dimension-v2}
\Dim(\A)=\sum_{i\in\Irr(\A)}\dim_\A(V_i)^2.
\end{equation}
A \defnfont{braiding} on a fusion category is a natural commutativity constraint
\[
c_{X,Y}:X\otimes Y\longrightarrow Y\otimes X
\]
satisfying the hexagon axioms. A \defnfont{ribbon structure} is a twist
on a braided fusion category \(\C\), that is, a natural automorphism
\(\theta:\operatorname{id}_\C\Rightarrow\operatorname{id}_\C\) with components
\(\theta_X:X\to X\), satisfying the balancing identity
\[
\theta_{X\otimes Y}
=(\theta_X\otimes\theta_Y)\circ c_{Y,X}\circ c_{X,Y}
\]
and the duality compatibility
\[
\theta_{X^*}=(\theta_X)^*.
\]
A ribbon fusion category comes with a spherical structure, and its dimensions
are the dimensions used in the ribbon graphical calculus; see
\cite[Definition~8.10.1]{EGNO2015} and
\cite[Chapter~I, Sections~1--2]{TuraevBook}.

\subsection{Modular tensor categories}
\label{subsec:modular-tensor-categories-v2}

Let \(\C\) be a ribbon fusion category. We write
\[
I=\Irr(\C)
\]
and choose representatives \(V_i\), \(i\in I\), with \(V_0=\mathbf 1\). The
unnormalized \(S\)-matrix is
\[
S=(S_{ij})_{i,j\in I},
\qquad
S_{ij}
=\operatorname{tr}_{V_i\otimes V_j}
\left(c_{V_j,V_i}\circ c_{V_i,V_j}\right),
\]
where the trace is the spherical trace coming from the ribbon structure. Thus
\(S_{ij}\) is the invariant of the Hopf link with components colored by
\(V_i\) and \(V_j\). The category \(\C\) is a \defnfont{modular category}, or
\defnfont{modular tensor category}, if this matrix \(S\) is invertible; see
\cite[Definitions~8.13.1--8.13.4]{EGNO2015},
\cite[Chapter~3, Definition~3.1.1]{BakalovKirillov2001},
and \cite[Chapter~II, Section~1.4]{TuraevBook}.
Since \(V_i\) is simple, the twist \(\theta_{V_i}\) is a scalar multiple of
\(\operatorname{id}_{V_i}\); we write this scalar as \(\theta_i\). We use the
unnormalized convention in which the quantum dimension is
\[
d_i:=S_{0i}.
\]
With the trace convention above, \(d_i\) is the spherical dimension of
\(V_i\); see \cite[Chapter~II, Section~1]{TuraevBook} and
\cite[Section~4.5.2]{TuraevVirelizier2017}. Thus
\begin{equation}
\label{eq:modular-global-dimension-v2}
\Dim(\C)=\sum_{i\in I}d_i^2.
\end{equation}
Following Turaev's normalization for 3-manifold invariants, whenever a modular
category is used as input for the Reshetikhin--Turaev invariant we also fix a
\defnfont{rank}, that is, a choice of scalar
\[
D_\C\in k^\times,\qquad D_\C^2=\Dim(\C).
\]
This choice is not canonical and is part of the fixed Reshetikhin--Turaev
normalization.
For a modular category the numbers \(d_i\) are nonzero, so the denominators
involving \(d_i\) below are nonzero; see
\cite[Chapter~II, Section~1.4]{TuraevBook} and
\cite[Lemma~4.2(a)]{TuraevVirelizier2017}.

The fusion coefficients of \(\C\) are recovered from \(S\) by the Verlinde
formula
\begin{equation}
\label{eq:verlinde-formula-v2}
N_{ij}^\ell
=
\frac{1}{\Dim(\C)}
\sum_{a\in I}\frac{S_{ia}S_{ja}S_{\ell^*a}}{d_a}.
\end{equation}
See \cite[Corollary~8.14.4]{EGNO2015}.

We call the following scalars the positive and negative \defnfont{Gauss sums}
of \(\C\):
\[
\Delta_+=\sum_{i\in I}\theta_i d_i^2,
\qquad
\Delta_-=\sum_{i\in I}\theta_i^{-1}d_i^2.
\]
For a modular category these scalars are non-zero and satisfy
\begin{equation}
\label{eq:gauss-product-v2}
\Delta_+\Delta_-=\Dim(\C).
\end{equation}
The \defnfont{anomaly} of \(\C\) is the scalar
\[
\alpha_\C:=\frac{\Delta_+}{\Delta_-}\in k^\times.
\]
We call \(\C\) \defnfont{anomaly-free} if \(\alpha_\C=1\), equivalently
\(\Delta_+=\Delta_-\). In that case we write
\[
\D:=\Delta_+=\Delta_-,
\]
and hence
\begin{equation}
\label{eq:anomaly-free-dimension-v2}
\D^2=\Dim(\C)=\sum_{i\in I}d_i^2.
\end{equation}
See \cite[Definition~8.15.1 and Proposition~8.15.4]{EGNO2015} and
\cite[Section~4.5.2]{TuraevVirelizier2017}.

Let \(\mathsf C=(\mathsf C_{ij})\) be the charge-conjugation matrix,
\(\mathsf C_{ij}=\delta_{i,j^*}\).
We use the modular-data identities
\begin{equation}
\label{eq:basic-modular-identities-v2}
S_{ij}=S_{ji},
\qquad
S_{i,j^*}=S_{i^*,j},
\qquad
d_{i^*}=d_i,
\qquad
S^2=(\Delta_+\Delta_-)\mathsf C=\Dim(\C)\mathsf C.
\end{equation}
For these modular-data identities and the Verlinde formula, see
\cite[Sections~8.13--8.16]{EGNO2015} and
\cite[Chapter~II, Section~3]{TuraevBook}.

The computational model requires the relevant scalars to lie in a number
field. Since \(k\) has characteristic zero, we identify the prime field with
\(\mathbb Q\subset k\). For a modular datum over \(k\), it is proved in
\cite{Etingof2009ModularData} that \(S_{ij}/d_j\) is an algebraic integer and
that the dimensions \(d_i\) are algebraic over \(\mathbb Q\); the twists have
finite order. Hence the scalars
\[
S_{ij},\qquad d_i,\qquad \theta_i,\qquad \Delta_\pm,
\qquad
\frac{S_{ij}}{d_i d_j}
\]
belong to a finite extension of \(\mathbb Q\) inside \(k\). After adjoining the
fixed rank \(D_\C\), if necessary, the same is true for the square root of
\(\Dim(\C)\). We also use the unitarity consequence proved there: after
embedding this finite extension in \(\mathbb C\), the numbers \(d_i\) are real,
\(\Dim(\C)=\sum_i d_i^2\) is positive, and the unnormalized \(S\)-matrix
satisfies
\[
SS^\dagger=\Dim(\C)I.
\]
See \cite{Etingof2009ModularData} and
\cite[Theorem~8.17.4 and Corollary~8.18.2]{EGNO2015}.

For the evaluation problems below, the category is always understood
together with a number field containing the relevant scalars. We fix a subfield
\[
k_0\subset k
\]
which is finite over \(\mathbb Q\), contains the structural constants and
normalization scalars needed to evaluate the relevant state sums or ribbon-link
invariants, and contains the modular-data scalars displayed above in the modular
case. The invariant values are regarded as elements of
\(k_0\). In the complexity-theoretic formulation we choose, once and for all, an
isomorphism from \(k_0\) onto a number field in \(\mathbb C\).

\subsection{RT and TV evaluation problems}
\label{subsec:rt-tv-evaluation-problems-v2}

The surgery input for the Reshetikhin--Turaev invariant is a framed-link
surgery presentation. Every closed connected oriented 3-manifold is obtained by
surgery on a framed link
\[
L\subset S^3.
\]
A surgery presentation may therefore be encoded by a framed-link diagram, together
with one integer framing on each component. If \(L\) has \(m\) components, let
\[
B_L\in M_m(\mathbb Z)
\]
be the linking matrix, with the framings on the diagonal, and let
\[
b_+(L),\qquad b_-(L),\qquad b_0(L)
\]
be the numbers of positive, negative, and zero eigenvalues of \(B_L\).

The ribbon graphical calculus of a ribbon category \(\C\) associates a scalar to
every \(\C\)-colored framed link in \(S^3\). If \(\C\) is modular, the
\defnfont{formal Kirby color} is the linear combination
\[
\Omega_\C=\sum_{i\in I} d_i V_i.
\]
If every component of \(L\) is colored by \(\Omega_\C\), we write
\[
\langle L(\Omega_\C)\rangle_\C
\]
for the corresponding ribbon-link invariant. The components of \(L\) may be
oriented arbitrarily: since \(d_{i^*}=d_i\), the formal color \(\Omega_\C\) is
invariant under duality, and the scalar is independent of these choices. For the
ribbon-graph invariant, the formal Kirby color, and the RT surgery normalization,
see \cite[Chapter~II, Sections~2--3]{TuraevBook} and
\cite[Theorems~4.1.12 and~4.1.16]{BakalovKirillov2001}. If \(M_L\) is the
closed connected oriented 3-manifold obtained by surgery on
\(L\), the chosen rank \(D_\C\) enters the Reshetikhin--Turaev normalization. In
the normalization used in this paper,
\begin{equation}
\label{eq:rt-surgery-normalization-v2}
Z_\C(M_L)
=
D_\C^{-b_0(L)-1}
\frac{\langle L(\Omega_\C)\rangle_\C}
{\Delta_+^{\,b_+(L)}\Delta_-^{\,b_-(L)}},
\qquad
D_\C^2=\Dim(\C).
\end{equation}
When \(\C\) is anomaly-free, we take the square root in this normalization to
be \(D_\C=\D\).
When Deligne products, reverse categories, and centers are used below, the
square-root choices are understood to be compatible with these constructions:
for the reverse modular category \(\C^{\mathrm{rev}}\), obtained from \(\C\) by
inverting the braiding and the twist, we use the same square root; for Deligne
products we use
\[
D_{\C\boxtimes\mathcal E}=D_\C D_{\mathcal E},
\]
and for centers we use the anomaly-free choice
\[
D_{\Z(\A)}=\Dim(\A).
\]
Here \(\Dim(\Z(\A))=\Dim(\A)^2\).
The square-root choices are taken compatibly with the RT TQFT normalizations;
see \cite[Chapter~IV, Sections~8--9]{TuraevBook} and
\cite[Corollaries~17.7--17.8]{TuraevVirelizier2017}.

We denote by
\[
\RTEVAL(\C)
\]
the function problem which takes as input a framed-link surgery
presentation of a closed connected oriented 3-manifold \(M\) and outputs
\(Z_\C(M)\).

For Turaev--Viro invariants, let \(\A\) be a spherical fusion category. The
Turaev--Viro invariant of a closed 3-manifold \(M\) is denoted by \(|M|_\A\).
The corresponding evaluation problem uses a triangulation of \(M\) as input; see
\cite[Theorem~13.1]{TuraevVirelizier2017}. We denote by
\[
\TVEVAL(\A)
\]
the function problem which takes as input a triangulation of a closed
connected oriented 3-manifold \(M\) and outputs \(|M|_\A\).

The comparison used below is through the Drinfeld center.
For a spherical fusion category \(\A\), its center
\[
\Z(\A)
\]
is an anomaly-free modular category; see
\cite[Theorem~5.4]{TuraevVirelizier2017}.
Moreover, the Turaev--Viro invariant of \(\A\) agrees with the
Reshetikhin--Turaev invariant of this center:
\begin{equation}
\label{eq:tv-rt-center-comparison-v2}
|M|_\A=Z_{\Z(\A)}(M)
\end{equation}
for every closed oriented 3-manifold \(M\); see
\cite[Theorem~17.1 and Corollary~17.7(a)]{TuraevVirelizier2017}.

\section{Graph partition functions and complexity}
\label{sec:graph-partition-complexity}

This section fixes the computational conventions and the graph-theoretic
dichotomy used in the reductions. The quantum invariants constructed below are
expressed as weighted graph homomorphism partition functions with algebraic
weights; the \(\#\mathrm P\)-hard cases are obtained from the bounded-degree
graph-homomorphism dichotomy of \cite{CaiGovorov2020}.

\subsection{Problems and reductions}
\label{subsec:exact-algebraic-computation}

All computational problems in this paper are exact evaluation problems.  The
algebraic datum, such as a category or a matrix, is fixed and is not part of the
input.  The input is only a finite combinatorial object: a framed-link surgery
presentation, a triangulation, or a graph.

The categories themselves are defined over the algebraically closed field \(k\).
For computational purposes, however, we work over a fixed number field \(K\)
containing the structural constants and normalization scalars needed for the
corresponding evaluation problem.

By a \defnfont{\(K\)-valued problem} we mean a function
\[
        F\colon \mathcal X\longrightarrow K,
\]
where \(\mathcal X\) is a class of finite combinatorial objects given by finite
encodings.  A \(K\)-valued problem is in
\(\FP\) if there is an algorithm which, on input \(x\in\mathcal X\), outputs the
exact encoding of \(F(x)\in K\) in time polynomial in the size of \(x\).

If \(F\colon\mathcal X\to K\) and \(G\colon\mathcal Y\to K\) are
\(K\)-valued problems, a \defnfont{polynomial-time Turing reduction}
\[
        F\leq_T^p G
\]
is an algorithm which computes \(F(x)\), for \(x\in\mathcal X\), in time
polynomial in the size of \(x\), using arithmetic in \(K\) and oracle calls
to \(G\) on inputs whose sizes are polynomially bounded in the size of \(x\).

Recall that \(\#\mathrm P\) is the class of integer-valued functions
\(f:\{0,1\}^\ast\to \mathbb N\) which count accepting computation paths of
nondeterministic polynomial-time machines \cite{Valiant1979}. The evaluation
problems considered in this paper are usually \(K\)-valued rather than
integer-valued, so we do not regard them as members of \(\#\mathrm P\).

A \(K\)-valued problem \(G\) is
\defnfont{\(\#\mathrm P\)-hard} if, for every \(f\in\#\mathrm P\), the
\(K\)-valued problem obtained from \(f\) by the inclusion
\[
        \mathbb Z\subset K
\]
admits a polynomial-time Turing reduction to \(G\).   We use this
oracle-reduction notion for \(K\)-valued evaluation problems throughout the paper
\cite{Papadimitriou1994,AroraBarak2009}.

\subsection{Weighted graph homomorphism partition functions}
\label{subsec:weighted-graph-homomorphism-functions-v2}

Let \(K\subset\mathbb C\) be a number field, let \(I\) be a finite set, and let
\[
A=(A_{ij})_{i,j\in I}\in M_I(K)
\]
be a fixed symmetric matrix. Graphs are finite and undirected. Unless explicitly
stated otherwise, they have no loops or multiple edges. For a graph \(G=(V,E)\),
define
\begin{equation}
\label{eq:graph-hom-partition-function-v2}
Z_A(G)=\sum_{\sigma:V\to I}\prod_{\{u,v\}\in E}A_{\sigma(u),\sigma(v)}.
\end{equation}
We refer to \(Z_A(G)\) as the weighted graph homomorphism partition function
defined by \(A\).
We write
\[
\EVAL(A)
\]
for the function problem which takes \(G\) as input and outputs
\(Z_A(G)\). For \(\Delta\geq1\), we write
\[
\EVAL^{(\Delta)}_{\ConnSimp}(A)
\]
for the restriction to connected simple graphs of maximum degree at most
\(\Delta\). This restricted form is used in the graph-manifold construction. See
\cite{CaiChenLu2013,CaiGovorov2020} for the
complex-valued graph-homomorphism dichotomy framework.

\subsection{A graph-homomorphism dichotomy}
\label{subsec:cai-govorov-v2}

We record the part of the Cai--Govorov dichotomy that will be used in the
reduction. The relevant algebraic condition is a multiplicative form of
block-rank-one. Only the square matrix case is needed here. For a matrix
\(A=(A_{ij})_{i,j\in I}\), its
support is
\[
\operatorname{supp}(A)=\{(i,j)\in I\times I:A_{ij}\neq0\}.
\]
The matrix \(A\) has \defnfont{rectangular support} if, after reordering rows
and columns separately, its zero--nonzero pattern is block diagonal, with each
nonzero block a full rectangle. Equivalently, there are pairwise disjoint
nonempty row sets \(R_1,\ldots,R_s\subseteq I\) and pairwise disjoint nonempty
column sets \(C_1,\ldots,C_s\subseteq I\), with any remaining rows or columns
zero, such that
\[
\operatorname{supp}(A)=
\bigsqcup_{p=1}^s R_p\times C_p.
\]
We allow \(s=0\), in which case \(A\) is the zero matrix.

If \(A\) has rectangular support, its \defnfont{nonzero support blocks} are the
submatrices
\[
A|_{R_p\times C_p},
\qquad 1\leq p\leq s.
\]
The matrix \(A\) is \defnfont{block-rank-one} if it has rectangular support and
each nonzero support block has ordinary matrix rank one. Equivalently, for each
\(p\), all \(2\times2\) minors inside \(R_p\times C_p\) vanish:
\[
A_{ij}A_{i'j'}=A_{ij'}A_{i'j}
\]
for all \(i,i'\in R_p\) and \(j,j'\in C_p\).

For \(r\geq1\), let
\[
A^{\circ r}=(A_{ij}^r)_{i,j\in I}
\]
be the \(r\)-th Hadamard power. Following \cite{CaiGovorov2020}, \(A\) is
\defnfont{multiplicative-block-rank-one} if \(A^{\circ r}\) is block-rank-one
for some \(r\geq1\). This condition is one of the algebraic conditions used in
the complex-valued graph-homomorphism dichotomy. The only consequence needed
below is the following implication proved by Cai--Govorov: if \(A\) is not
multiplicative-block-rank-one, then the evaluation problem \(G\mapsto
Z_A(G)\) is \(\#\mathrm P\)-hard on graphs of bounded degree. We use this
implication together with their extension to simple graphs.

The manifolds constructed in Section~\ref{sec:graph-manifold-family-v2} are
indexed by connected simple graphs. The restriction to connected graphs follows
from multiplicativity over connected components. If
\[
G=G_1\sqcup\cdots\sqcup G_m
\]
is the decomposition of \(G\) into connected components, then
\begin{equation}
\label{eq:graph-partition-components-v2}
Z_A(G)=\prod_{\alpha=1}^m Z_A(G_\alpha).
\end{equation}
Indeed, a labelling of \(G\) is the same as a collection of labellings of its
components, and the edge product in
\eqref{eq:graph-hom-partition-function-v2} factors accordingly. Hence an oracle
for connected simple graphs of maximum degree at most \(\Delta\) computes the
same partition function on all simple graphs of maximum degree at most
\(\Delta\), by first decomposing the input graph into connected components. Each
isolated vertex contributes the fixed scalar \(|I|\), and the empty graph
contributes \(1\).

We use the following bounded-degree consequence of the Cai--Govorov criterion
for non-multiplicative-block-rank-one matrices, together with their
extension of the dichotomy to simple graphs
\cite[Definition~3.4 and Theorem~20.4]{CaiGovorov2020}.

\begin{theorem}
\label{thm:cai-govorov-v2}
Let \(K\subset\mathbb C\) be a number field, let \(I\) be a finite set, and let
\[
A\in M_I(K)
\]
be a fixed symmetric matrix. If \(A\) is not multiplicative-block-rank-one, then
there exists a constant \(\Delta=\Delta(A)\) such that
\[
\EVAL^{(\Delta)}_{\ConnSimp}(A)
\]
is \(\#\mathrm P\)-hard under polynomial-time Turing reductions.
\end{theorem}

\section{The graph-manifold family}
\label{sec:graph-manifold-family-v2}

We associate to each finite connected simple graph \(G\) a closed oriented
3-manifold \(M_G\).  It is obtained by gluing product Seifert pieces along
boundary tori.  These boundary tori are parametrized as part of the
construction, and the parametrizations fix the coordinates used in the TQFT
computation of Section~\ref{sec:rt-formula-v2}.

The same manifolds can also be described as boundaries of four-dimensional
plumbings.  This plumbing viewpoint is used below to produce effective surgery
presentations; compare \cite[Chapter~X, Section~9]{TuraevBook} and
\cite[Example~4.6.2 and Section~6.1]{GompfStipsicz1999}.

\subsection{Half-edges and vertex blocks}
\label{subsec:local-blocks-v2}

Throughout this section \(G=(V(G),E(G))\) is a finite connected simple graph
with at least one edge. Its set of \defnfont{half-edges}, or incidences, is
\[
H(G)=\{(v,e):v\in V(G),\ e\in E(G),\ v\in e\}.
\]
For \(v\in V(G)\), put
\[
H(v)=\{(v,e)\in H(G):e\ni v\},
\qquad
\deg(v)=|H(v)|.
\]
The maximum degree of \(G\) is denoted by
\[
\Delta(G)=\max_{v\in V(G)}\deg(v).
\]
If \(e=\{u,v\}\), we write its two half-edges as
\[
h_{u,e}=(u,e),
\qquad
h_{v,e}=(v,e).
\]

For a finite set \(J\), write \(\Sigma_{1,J}\) for a compact connected
oriented genus-one surface whose boundary components are indexed by \(J\).
For each vertex \(v\), choose a compact connected oriented surface
\[
\Sigma_v=\Sigma_{1,H(v)}
\]
of genus one whose boundary components are indexed by the half-edges incident
to \(v\). For \(h\in H(v)\), let \(C_h\) denote the boundary component
corresponding to \(h\); thus each \(C_h\) is an oriented circle and
\[
\partial\Sigma_v=\bigsqcup_{h\in H(v)} C_h.
\]

The \defnfont{vertex block} corresponding to \(v\) is
\[
P_v=\Sigma_v\times S^1,
\]
with the product orientation. Its boundary is a disjoint union of tori
\[
\partial P_v=\bigsqcup_{h\in H(v)}T_h,
\qquad
T_h=C_h\times S^1.
\]
On \(T_h\) we distinguish the two circle directions
\[
\mu_h=C_h\times\{*\},
\qquad
\lambda_h=\{*\}\times S^1.
\]
Here \(\mu_h\) is the boundary-circle direction of the base surface and
\(\lambda_h\) is the product \(S^1\)-direction.

We compare all boundary tori with one fixed reference torus. Let
\[
T^2_{\mathrm{std}}=S^1_\mu\times S^1_\lambda
\]
be the standard torus, where \(S^1_\mu\) and \(S^1_\lambda\) are two oriented
copies of \(S^1\). We equip it with the ordered homology basis
\[
([S^1_\mu],[S^1_\lambda]).
\]
A \defnfont{parametrization} of a boundary torus \(T_h\) means an
orientation-preserving homeomorphism from this fixed reference torus to \(T_h\).
For every half-edge \(h\), fix such a parametrization
\[
p_h:T^2_{\mathrm{std}}\longrightarrow T_h
\]
which sends the oriented homology classes \([S^1_\mu]\) and
\([S^1_\lambda]\) to \([\mu_h]\) and \([\lambda_h]\), respectively.
These parametrizations are part of the construction.

\subsection{Boundary parametrizations and edge gluings}
\label{subsec:edge-gluings-v2}

Fix an orientation-preserving self-homeomorphism of \(T^2_{\mathrm{std}}\), also
denoted by \(s\), whose induced action on \(H_1(T^2_{\mathrm{std}};\mathbb Z)\)
is
\[
s_*=
\begin{pmatrix}
0&-1\\
1&0
\end{pmatrix}
\in \operatorname{SL}(2,\mathbb Z).
\]
Fix also an orientation-reversing involution
\[
\rho:T^2_{\mathrm{std}}\longrightarrow T^2_{\mathrm{std}},
\qquad
\rho_*=
\begin{pmatrix}
1&0\\
0&-1
\end{pmatrix}
\]
in the ordered basis \(([S^1_\mu],[S^1_\lambda])\).
If \(T_h\) is an oriented torus, then \(-T_h\) denotes the same underlying torus
with the opposite orientation. The orientation-preserving parametrization of the
opposite torus is
\[
p_h^-:=p_h\circ\rho:T^2_{\mathrm{std}}\longrightarrow -T_h.
\]
Here \(p_h\circ\rho\) is first the underlying continuous map to \(T_h\); in the
displayed formula the target is regarded with the opposite orientation. Since
\(\rho\) reverses orientation and the target orientation has also been reversed,
\(p_h^-\) is orientation-preserving.

For an edge \(e=\{u,v\}\), let
\[
h=h_{u,e},
\qquad
\bar h=h_{v,e}
\]
be its two half-edges. The \defnfont{edge-gluing map} associated with
\(e\) is the homeomorphism
\[
g_e:T_h\longrightarrow -T_{\bar h}
\]
defined by the commutative diagram
\[
\begin{CD}
T^2_{\mathrm{std}} @>{p_h}>> T_h\\
@V{s}VV @VV{g_e}V\\
T^2_{\mathrm{std}} @>{p_{\bar h}^-}>> -T_{\bar h}.
\end{CD}
\]
Thus
\begin{equation}
\label{eq:graph-edge-gluing-map-v2}
g_e=p_{\bar h}^-\circ s\circ p_h^{-1}:T_h\longrightarrow -T_{\bar h}.
\end{equation}
As an underlying map \(T_h\to T_{\bar h}\), it acts on first homology by
\begin{equation}
\label{eq:edge-gluing-homology-action-v2}
(g_e)_*[\mu_h]=-[\lambda_{\bar h}],
\qquad
(g_e)_*[\lambda_h]=-[\mu_{\bar h}].
\end{equation}
The signs come from the boundary orientations.

The closed 3-manifold associated with \(G\) is obtained from
\[
P_G^\circ=\bigsqcup_{v\in V(G)}P_v
\]
by gluing, for every edge \(e=\{u,v\}\), the torus \(T_{h_{u,e}}\) to
\(-T_{h_{v,e}}\) using \(g_e\). We denote the quotient by
\[
M_G.
\]

\begin{proposition}
\label{prop:MG-closed-connected-v2}
For every finite connected simple graph \(G\) with at least one edge, \(M_G\) is
a closed connected oriented 3-manifold.
\end{proposition}

\begin{proof}
Every boundary torus of \(P_G^\circ\) is indexed by a half-edge, and every
half-edge belongs to a unique edge of \(G\). Hence every boundary component is
glued once. The maps \(g_e:T_h\to -T_{\bar h}\) are
orientation-preserving as maps to the oppositely oriented target. Hence the
orientations of the vertex blocks glue to an orientation of the quotient. Since
the blocks \(P_v\) are connected and the incidence graph of the gluing is
\(G\), connectedness of \(G\) implies connectedness of \(M_G\).
\end{proof}

\subsection{Plumbing and effective presentations}
\label{subsec:effective-presentations-v2}

The construction above is arranged so that \(M_G\) is the boundary of a
plumbed 4-manifold.  We use this plumbing description as an effective way to
produce the input presentations needed later: a framed-link surgery presentation
for Reshetikhin--Turaev evaluation and, from it, a triangulation for
Turaev--Viro evaluation.  The plumbing and Kirby-calculus conventions are those
of \cite[Example~4.6.2 and Sections~5.4,~6.1]{GompfStipsicz1999}.

Following Gompf--Stipsicz, let \(\Gamma_G\) be the oriented decorated plumbing
graph whose underlying graph is \(G\), with all vertex labels
\[
        (g_v,e_v)=(1,0)
\]
and with all edge signs positive.  Thus each vertex represents the oriented
\(D^2\)-bundle over \(T^2\) with Euler number zero, and each edge represents one
positive plumbing between the corresponding disk bundles.  In this family all
genera, Euler numbers, and plumbing signs are fixed; only the combinatorics of
\(G\) varies.

Gompf--Stipsicz give Kirby-calculus rules which, starting from a decorated
plumbing graph, produce a Kirby diagram for the corresponding plumbed
4-manifold.  We apply these rules to \(\Gamma_G\).  Choose a rooted spanning
tree
\[
        (\mathcal T,r)
\]
of \(G\).  The root \(r\) is used to order the edges of \(\mathcal T\): we
choose an order in which each edge attaches one new vertex to the part of the
tree already constructed.  Write
\[
        E_{\mathrm{cyc}}=E(G)\setminus E(\mathcal T).
\]
For each vertex \(v\), take one copy \(B_v\) of the Kirby diagram for
\(T^2\times D^2\), shown in Figure~\ref{fig:borromean-vertex-block}.  The two
local moves used to join these copies are shown in
Figure~\ref{fig:plumbing-moves}.  The Type I move is used for the edges of the
spanning tree \(\mathcal T\).  Thus, following the chosen order, each edge
\(e=\{u,v\}\in E(\mathcal T)\) is realized by applying the Type I move to the
corresponding copies \(B_u\) and \(B_v\).  For each remaining edge
\(e=\{u,v\}\in E_{\mathrm{cyc}}\), choose an auxiliary orientation, say from
\(u\) to \(v\), and realize \(e\) by applying the Type II move to the ordered
pair \(B_u,B_v\), as in the nonsimply connected plumbing move of
\cite[Figures~6.7--6.8]{GompfStipsicz1999}.  The simplicity assumption excludes
self-plumbings.

\begin{figure}[H]
\centering
\includegraphics[width=.56\textwidth]{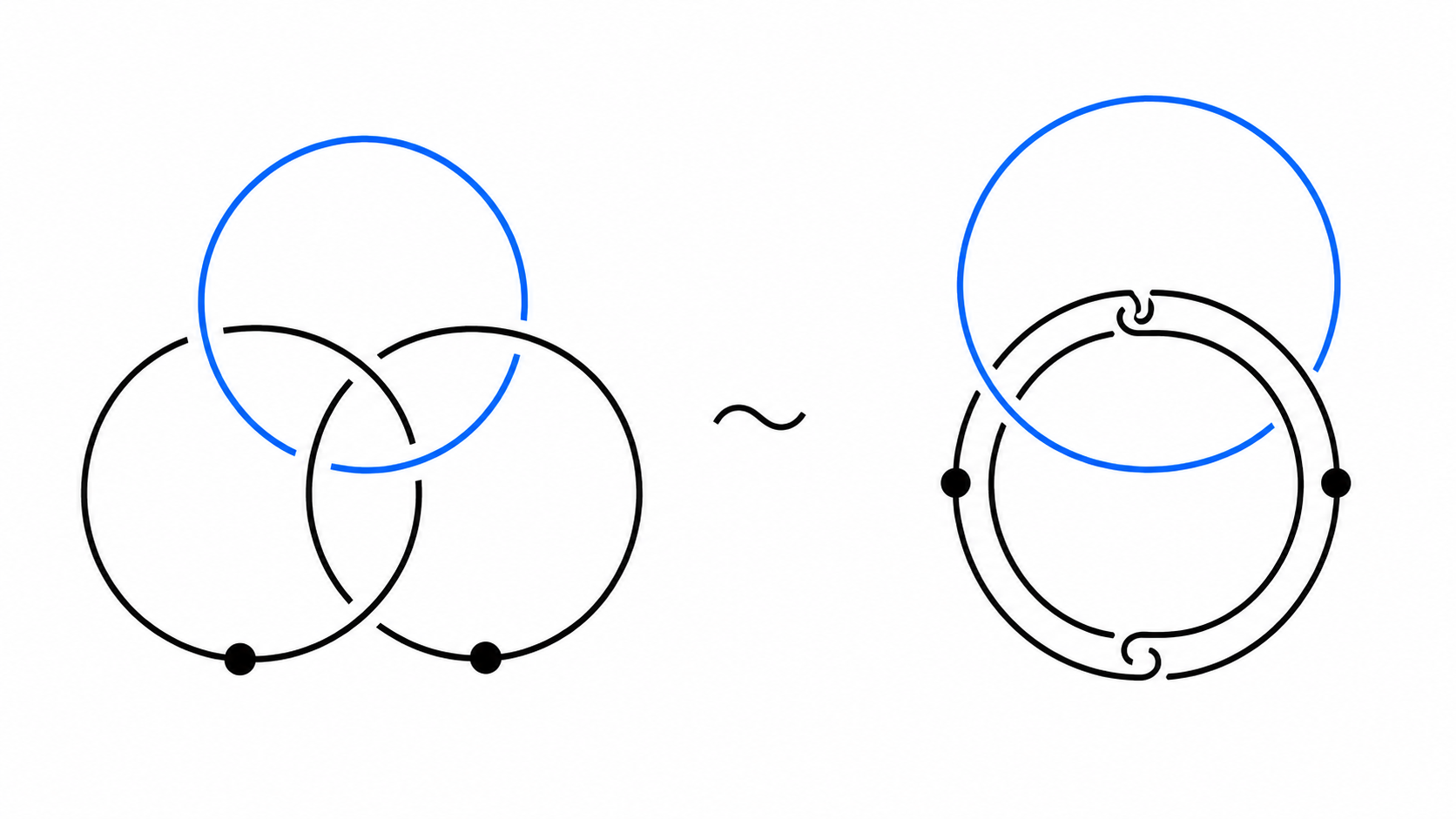}
\caption{The Kirby diagram for \(T^2\times D^2\), used as the vertex block.
The two dotted components represent the two \(1\)-handles of the torus base,
and the remaining \(2\)-handle component is \(0\)-framed.}
\label{fig:borromean-vertex-block}
\end{figure}

\begin{figure}[H]
\centering
\begin{minipage}{.48\textwidth}
\centering
\includegraphics[width=\textwidth,trim=95 105 95 105,clip]{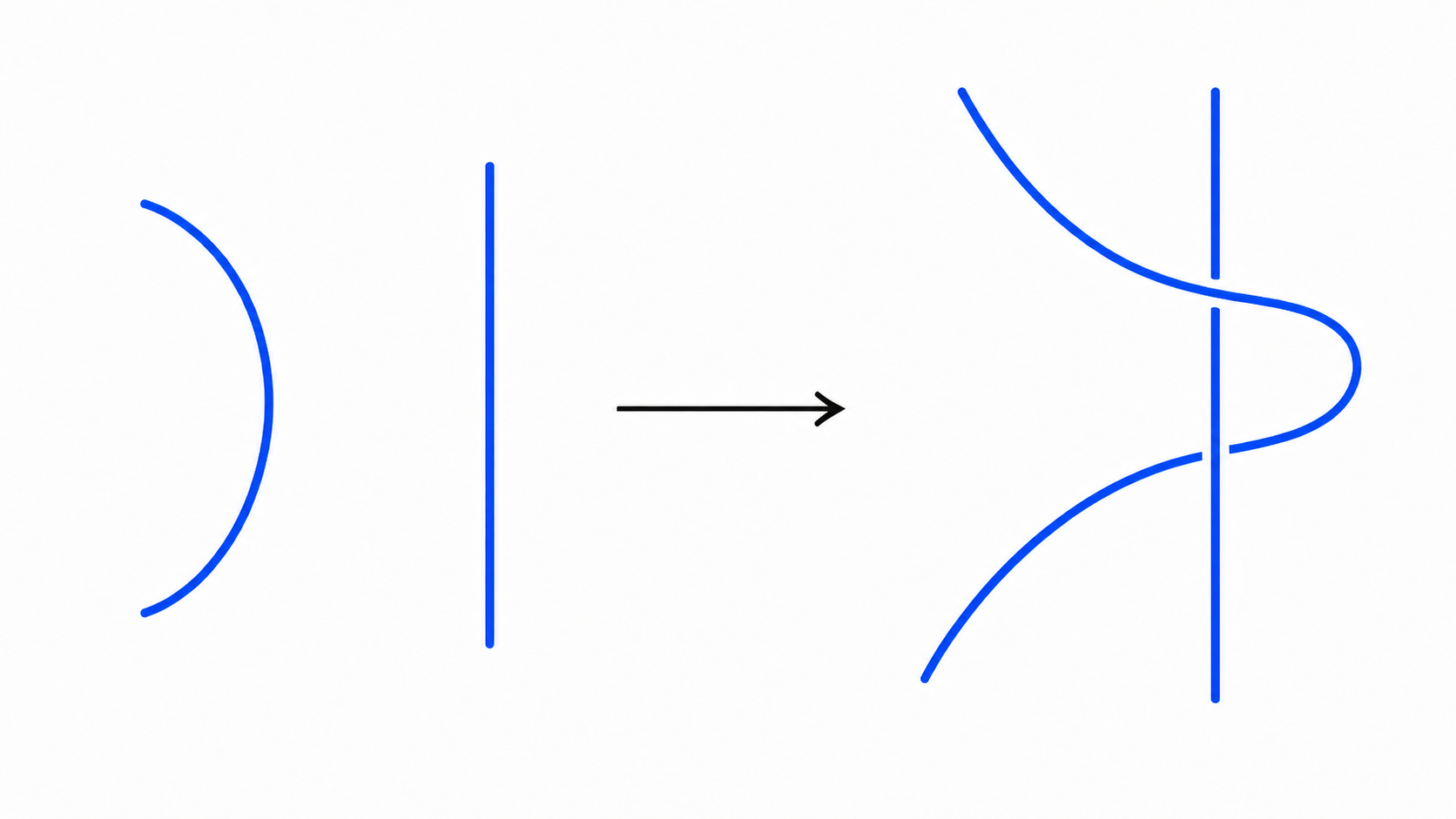}
\end{minipage}\hfill
\begin{minipage}{.48\textwidth}
\centering
\includegraphics[width=\textwidth,trim=95 105 95 105,clip]{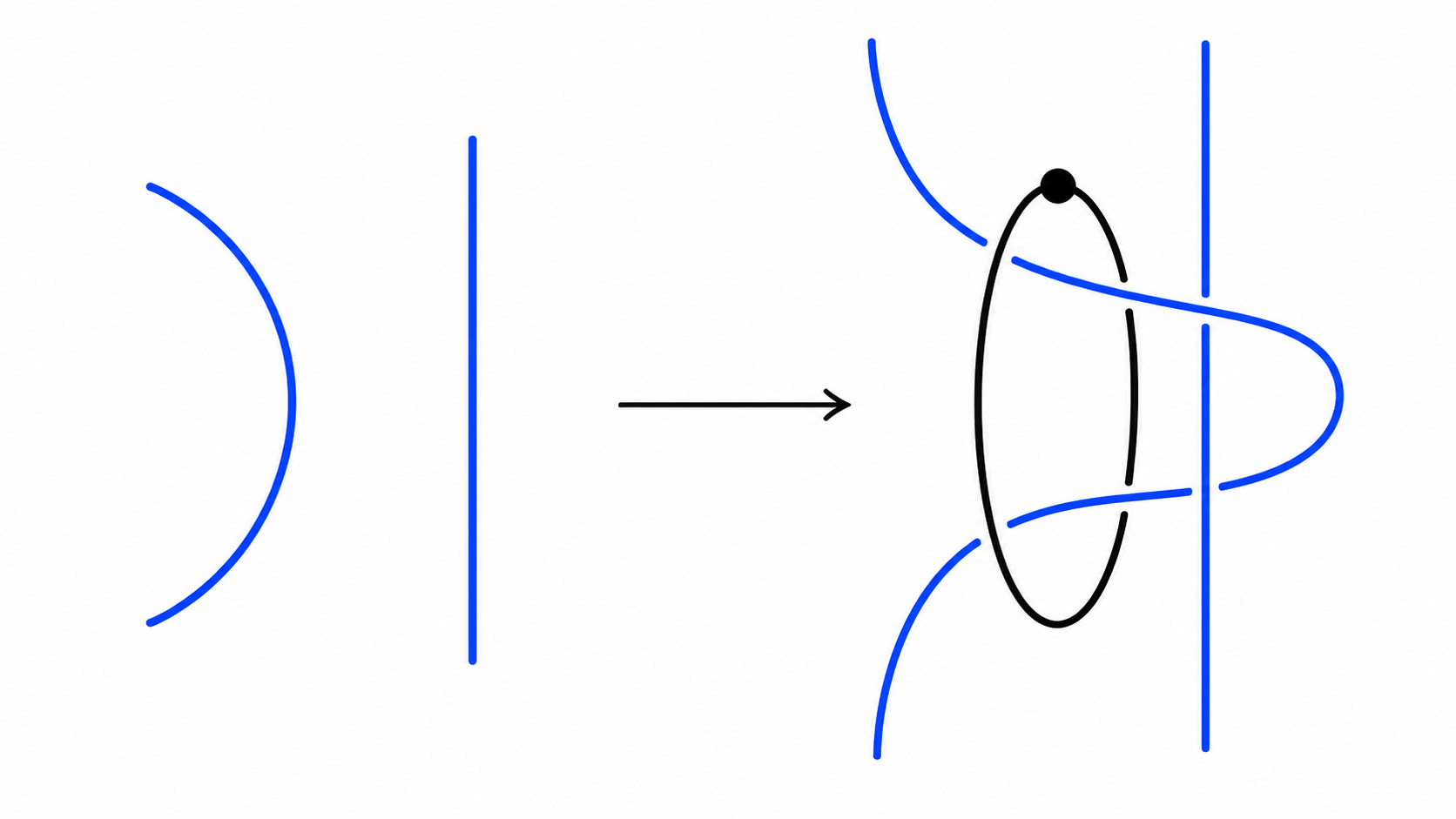}
\end{minipage}
\caption{Local plumbing moves: Type I for the edges of \(\mathcal T\), and
Type II for the edges in \(E_{\mathrm{cyc}}\) after choosing an orientation.}
\label{fig:plumbing-moves}
\end{figure}

After all local moves have been applied, the result is a Kirby diagram for the
plumbed 4-manifold associated with \(\Gamma_G\). Replace each dotted component,
following the \(1\)-handle notation of
\cite[Section~5.4]{GompfStipsicz1999}, by a \(0\)-framed surgery component. We
denote the resulting planar framed-link diagram by \(D_G\), and the framed link it
represents by
\[
        L_G\subset S^3.
\]
In the present family every component of \(L_G\) is \(0\)-framed.

\begin{proposition}
\label{prop:MG-plumbing-model-v2}
The manifold \(M_G\) is orientation-preservingly diffeomorphic to the boundary
of the four-dimensional plumbing associated with \(\Gamma_G\).
\end{proposition}

\begin{proof}
In the plumbing model of oriented disk bundles over surfaces
\cite[Example~4.6.2 and Section~6.1]{GompfStipsicz1999}, a vertex \(v\)
corresponds to the oriented \(D^2\)-bundle over \(T^2\) with
Euler number zero.  For each half-edge incident to \(v\), choose a small disk in
the base and remove the corresponding local product \(D^2\times D^2\) from the
total space.  The vertical part of the boundary of the remaining piece is
\[
        \Sigma_{1,H(v)}\times S^1.
\]
On the boundary torus corresponding to a half-edge \(h\), the circle
\(C_h\times\{*\}\) is the base-boundary direction and \(\{*\}\times S^1\) is
the fiber direction.  A positive plumbing operation between two such local
products interchanges the base and fiber directions.  With the boundary
orientations used above, if \(h\) and \(\bar h\) are the two half-edges of the
plumbing edge, the underlying boundary map sends
\[
        \mu_h\longmapsto -\lambda_{\bar h},
        \qquad
        \lambda_h\longmapsto -\mu_{\bar h}.
\]
The displayed boundary map agrees with
\eqref{eq:edge-gluing-homology-action-v2}. Therefore the boundary of the
plumbed 4-manifold is obtained by gluing the boundary blocks
\[
        P_v=\Sigma_{1,H(v)}\times S^1
\]
using the maps \(g_e\) defined in \eqref{eq:graph-edge-gluing-map-v2}. Hence the
boundary of the plumbing is diffeomorphic to \(M_G\).
\end{proof}

We illustrate the construction on the graph \(G_0\) shown in
Figure~\ref{fig:example-graph}.

\begin{figure}[H]
\centering
\begin{tikzpicture}[
  scale=1.0,
  vertex/.style={circle,fill=black,inner sep=2.2pt},
  edge/.style={line width=.8pt}
]
\node[vertex,label=left:{\small \(1\)}] (x1) at (-1.8,0) {};
\node[vertex,label=below:{\small \(2\)}] (x2) at (0,0) {};
\node[vertex,label=above:{\small \(3\)}] (x3) at (1.6,1.1) {};
\node[vertex,label=below:{\small \(4\)}] (x4) at (1.6,-1.1) {};
\draw[edge] (x1) -- (x2);
\draw[edge] (x2) -- (x3);
\draw[edge] (x2) -- (x4);
\draw[edge] (x3) -- (x4);
\end{tikzpicture}
\caption{The example graph \(G_0\).}
\label{fig:example-graph}
\end{figure}
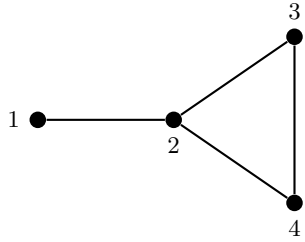

Choose the spanning tree with edges \(\{1,2\}\), \(\{2,3\}\), and
\(\{2,4\}\), root it at vertex \(1\), and process the edges in this order,
applying the Type I move each time.  The remaining edge is \(\{3,4\}\); for the
Type II move, orient it from \(4\) to \(3\). Applying these four local moves gives
the Kirby diagram in
Figure~\ref{fig:example-kirby-diagram}.

\begin{figure}[H]
\centering
\includegraphics[width=.44\textwidth,trim=120 0 450 0,clip]{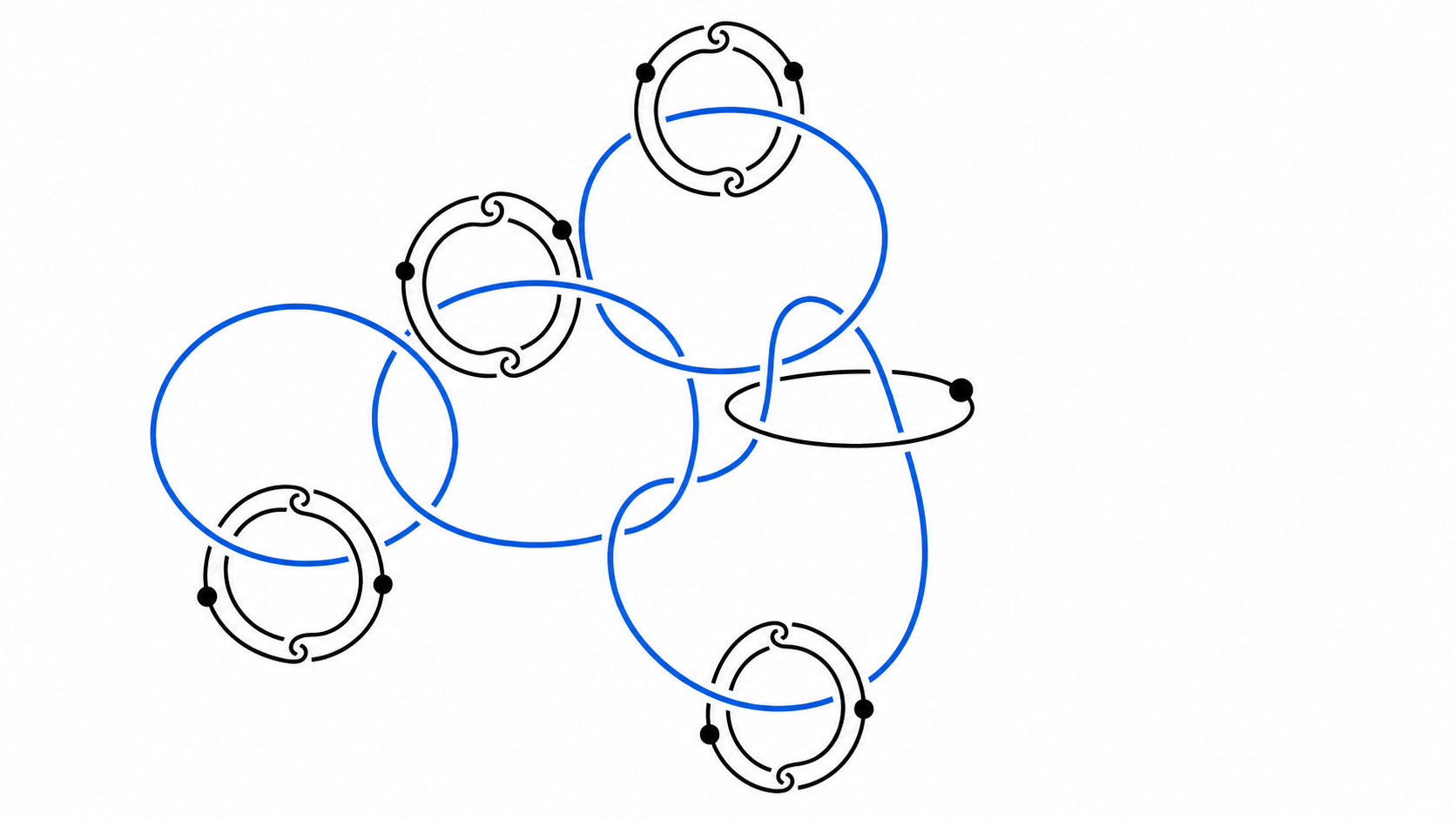}
\caption{The Kirby diagram for \(G_0\).}
\label{fig:example-kirby-diagram}
\end{figure}

Finally, replace every dotted circle by a \(0\)-framed surgery component.  In
the picture this is shown by removing the dot markings and keeping the same link
diagram; the result is the planar surgery diagram in
Figure~\ref{fig:example-surgery-link}.

\begin{figure}[H]
\centering
\includegraphics[width=.44\textwidth,trim=120 0 450 0,clip]{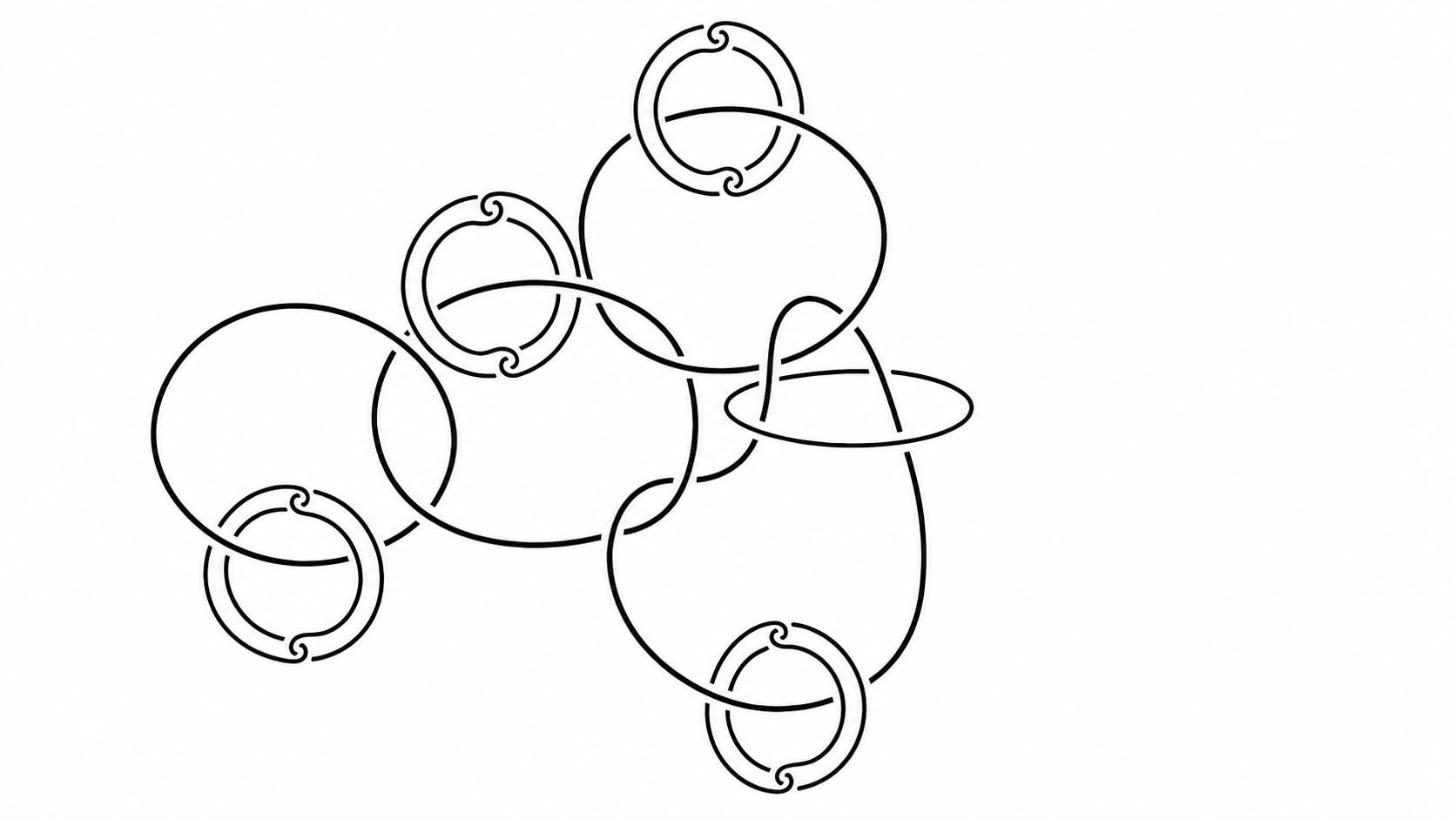}
\caption{The framed surgery link obtained from
Figure~\ref{fig:example-kirby-diagram}.}
\label{fig:example-surgery-link}
\end{figure}

\begin{proposition}
\label{prop:effective-surgery-v2}
There is a polynomial-time algorithm which, given a finite connected simple graph
\(G\) with at least one edge, outputs a planar framed-link diagram \(D_G\)
representing a framed link \(L_G\subset S^3\), all of whose components are
\(0\)-framed, such that surgery on \(L_G\) presents \(M_G\):
\[
        S^3_{L_G}\cong M_G.
\]
\end{proposition}

\begin{proof}
A spanning tree \(\mathcal T\subset G\), a root, and a traversal order of the
edges of \(\mathcal T\) can be found in time \(O(|V(G)|+|E(G)|)\), for example
by breadth-first search or depth-first search.  The remaining choices in the
construction are auxiliary orientations of the edges in
\(E(G)\setminus E(\mathcal T)\).  The construction then applies one fixed local
diagram for each vertex and each edge of \(G\).  These local diagrams can be
assembled in the plane with polynomially many crossings; hence the resulting
planar Kirby diagram, and therefore the framed-link diagram
\(D_G\), have size polynomial in \(|V(G)|+|E(G)|\).

Before replacing dotted circles, the diagram is the standard Kirby diagram for
the plumbing associated with \(\Gamma_G\).  By
Proposition~\ref{prop:MG-plumbing-model-v2}, the boundary of this plumbing is
orientation-preservingly diffeomorphic to \(M_G\).  Replacing all dotted
components by \(0\)-framed surgery components gives the planar framed-link
diagram \(D_G\).  The resulting framed link \(L_G\subset S^3\) has surgery
manifold equal to the boundary of the handlebody represented by the Kirby
diagram.  Therefore
\[
        S^3_{L_G}\cong M_G.
\]
\end{proof}

For the Turaev--Viro input we use ordered triangulations, that is,
triangulations whose finite vertex set is equipped with a total order.  We do
not need a triangulation adapted to the graph-manifold decomposition; a
polynomial-size triangulation follows from the surgery presentation of
Proposition~\ref{prop:effective-surgery-v2}.

\begin{proposition}
\label{prop:effective-triangulation-v2}
There is a polynomial-time algorithm which, given a finite connected simple graph
\(G\) with at least one edge, outputs an ordered triangulation \(t_G\) of
\(M_G\) whose size is polynomial in \(|V(G)|+|E(G)|\).
\end{proposition}

\begin{proof}
Let \(D_G\) be the planar framed-link diagram produced by
Proposition~\ref{prop:effective-surgery-v2}, and let \(L_G\) be the framed link
it represents.  Let \(c\) be the number of crossings of this displayed diagram.
By the construction in Proposition~\ref{prop:effective-surgery-v2}, \(c\) is
polynomial in \(|V(G)|+|E(G)|\), every component is \(0\)-framed, every
component is involved in a crossing, and \(D_G\) has no local kink.  Here a
local kink means a small one-crossing curl on a component, removable by a
Reidemeister-I move.

We use \cite[Lemma~2.1]{Cha2018}, which states that if a framed link has a
planar diagram with \(c\) crossings, no local kink, and every zero-framed
component involved in a crossing, and if \(w_i\) denotes the writhe of the
\(i\)-th component, then the surgery manifold admits a triangulation with at
most
\[
        96c+48\sum_i |f_i-w_i|
\]
tetrahedra, where \(f_i\) is the surgery framing.  The construction in the proof
of the lemma produces such a triangulation; in our case \(f_i=0\) for every
component, and
\[
        \sum_i |w_i|\leq c
\]
because each crossing contributes at most one unit to the writhe of one
component.  Thus the construction in \cite[Lemma~2.1]{Cha2018} gives a
triangulation of \(S^3_{L_G}\) with at most \(144c\) tetrahedra, and produces it
in polynomial time from \(D_G\).  Since \(S^3_{L_G}\cong M_G\), this is a
polynomial-size triangulation of \(M_G\).  Choose any total order on its finite
vertex set.
\end{proof}

\section{The RT formula for the graph manifolds}
\label{sec:rt-formula-v2}

\subsection{TQFT notation and gluing}
\label{subsec:tqft-notation-gluing-v2}

We use the anomaly-free Reshetikhin--Turaev TQFT associated with \(\C\), with
the convention that its value on a closed 3-manifold agrees with the surgery
normalization recalled in Section~\ref{subsec:rt-tv-evaluation-problems-v2}. We
recall only the notation needed for the calculation; for the full
construction see \cite[Chapter~III, Section~1 and Chapter~IV,
Sections~1.1--1.8]{TuraevBook}.
The terminology of decorated surfaces, opposite surfaces, items, decorated
cobordisms, and cylinders is that of
\cite[Chapter~IV, Sections~1.1--1.6]{TuraevBook}.

An \defnfont{item} is a pair \((W,\varepsilon)\), where \(W\) is an object of
\(\C\) and \(\varepsilon\in\{+1,-1\}\). A \defnfont{decorated surface} is a
closed oriented surface together with a finite ordered family of pairwise
disjoint oriented arcs, each equipped with an item. The unmarked case is allowed
and will be used for the boundary tori. A decorated homeomorphism preserves the
orientation, the marked arcs, their order, and their items. If \(\Sigma\) is a
decorated surface, then \(-\Sigma\) denotes the opposite decorated surface:
the orientation of the underlying surface and of every marked arc is reversed,
and each item \((W,\varepsilon)\) is replaced by \((W,-\varepsilon)\).
We write
\[
W^{+1}=W,\qquad W^{-1}=W^*.
\]
The label \(W\) itself is therefore not replaced by a dual object when the opposite
surface is formed; the dual object enters through the sign of the item.

A \defnfont{parametrized decorated surface} is a decorated surface together
with a chosen decorated homeomorphism from a model surface of the same type.
For the unmarked torus this is the parametrization fixed in
Section~\ref{subsec:local-blocks-v2},
\[
p_h:T^2_{\mathrm{std}}\longrightarrow T_h.
\]
We use this parametrization to identify the state space of \(T_h\) with the
state space of the standard torus.

To a parametrized decorated surface \(\Sigma\), the TQFT assigns a
finite-dimensional vector space
\[
Z_\C(\Sigma).
\]
The opposite surface is dual in the TQFT sense: there is a non-degenerate
evaluation pairing
\[
d_\Sigma:Z_\C(\Sigma)\otimes Z_\C(-\Sigma)\longrightarrow k.
\]

A \defnfont{decorated cobordism}
\[
W:\Sigma_-\longrightarrow\Sigma_+
\]
consists of a compact oriented 3-manifold, an identification of its boundary
with
\[
(-\Sigma_-)\sqcup \Sigma_+,
\]
and a \(\C\)-colored ribbon graph in the 3-manifold. Here a ribbon graph
means a framed embedded graph with colored oriented bands and coupon vertices;
edge colors are objects of \(\C\), and coupon colors are morphisms in \(\C\).
The free boundary ends of the ribbon graph lie on the marked arcs of
\(\Sigma_-\) and \(\Sigma_+\), with labels matching the corresponding items.
When no ribbon graph is present, we suppress it from the notation.

Fix Lagrangians on the reference boundary tori and give each local cobordism
weight zero. For the gluings below the corresponding
Maslov corrections are trivial; since \(\C\) is anomaly-free, no further anomaly
factor appears in the gluing or trace formulas. See
\cite[Chapter~IV, Sections~8--9]{TuraevBook}.

The anomaly-free RT theory is a symmetric monoidal functor from this decorated
cobordism category to finite-dimensional \(k\)-vector spaces. Hence a decorated
cobordism \(W:\Sigma_-\to\Sigma_+\) gives a linear map
\[
Z_\C(W):Z_\C(\Sigma_-)\longrightarrow Z_\C(\Sigma_+).
\]
In particular, a cobordism \(W:\varnothing\to\Sigma\) gives a vector
\[
Z_\C(W)\in Z_\C(\Sigma),
\]
and a closed 3-manifold gives a scalar. The monoidal structure identifies
\[
Z_\C(\Sigma_1\sqcup\Sigma_2)
\cong
Z_\C(\Sigma_1)\otimes Z_\C(\Sigma_2),
\]
and similarly for finite disjoint unions.

If \(W_1:\Sigma_0\to\Sigma_1\) and \(W_2:\Sigma_1\to\Sigma_2\) are composable
cobordisms, then
\[
Z_\C(W_2\circ W_1)=Z_\C(W_2)\circ Z_\C(W_1).
\]
The decorated cylinder \(\Sigma\times[0,1]\) represents the identity map on
\(Z_\C(\Sigma)\).

If \(f:\Sigma\to\Sigma'\) is a decorated homeomorphism, functoriality applied to
the mapping cylinder of \(f\) gives an isomorphism
\[
f_\#:Z_\C(\Sigma)\longrightarrow Z_\C(\Sigma').
\]
For unmarked tori, orientation-preserving homeomorphisms are decorated
homeomorphisms. In formulas below we use
the maps \((p_h)_\#\) to transport the standard torus basis to \(Z_\C(T_h)\),
and the inverse maps to identify \(Z_\C(T_h)\) with
\(Z_\C(T^2_{\mathrm{std}})\). After these identifications are fixed, we
suppress the parametrization maps from the notation.

We use the following RT gluing formula. Suppose
\[
W_1:\varnothing\to \Sigma,
\qquad
W_2:\varnothing\to \Sigma',
\]
and suppose that \(f:\Sigma\to-\Sigma'\) is the boundary homeomorphism used to
glue \(W_1\) to \(W_2\). Then
\begin{equation}
\label{eq:tqft-gluing-pairing-v2}
Z_\C(W_1\cup_f W_2)
=
d_{\Sigma'}\bigl(Z_\C(W_2)\otimes f_\#Z_\C(W_1)\bigr).
\end{equation}
More generally, when several boundary components are glued, the invariant is
obtained by applying this pairing independently on each glued pair of boundary
state spaces. We also use the trace formula for products: if \(X\) is a
decorated closed surface, then
\begin{equation}
\label{eq:tqft-trace-product-v2}
Z_\C(X\times S^1)=\dim Z_\C(X).
\end{equation}
More generally, if \(f:X\to X\) is a decorated homeomorphism and \(M_f\) is its
mapping torus, then
\begin{equation}
\label{eq:tqft-mapping-torus-trace-v2}
Z_\C(M_f)=\operatorname{Tr}(f_\#).
\end{equation}
These conventions are consequences of the RT TQFT axioms; see
\cite[Chapter~III, Theorems~2.1.1--2.1.3 and Chapter~IV,
Sections~1 and~5]{TuraevBook}.

\subsection{Torus states and local factors}
\label{subsec:tqft-input-v2}

Let
\[
\mathcal H_T=Z_\C(T^2_{\mathrm{std}})
\]
be the state space of the standard torus. We use the basis
\[
\{e_i\}_{i\in I}
\]
of \(\mathcal H_T\), defined as follows. View \(T^2_{\mathrm{std}}\) as the
boundary of the standard solid torus \(D^2\times S^1\). The vector \(e_i\) is the
state of the cobordism
\[
D^2\times S^1:\varnothing\longrightarrow T^2_{\mathrm{std}}
\]
with the core circle \(\{0\}\times S^1\) colored by \(V_i\). Under the standard
identification
\[
\mathcal H_T
\cong
\Hom_\C\left(\mathbf 1,\bigoplus_{j\in I} V_j\otimes V_j^*\right),
\]
this state corresponds to the coevaluation
\(\mathbf 1\to V_i\otimes V_i^*\) in the \(i\)-summand and to zero in the other
summands. Let
\[
\{e^i\}_{i\in I}\subset Z_\C(-T^2_{\mathrm{std}})
\]
be the dual basis, normalized by
\[
d_{T^2_{\mathrm{std}}}(e_i\otimes e^j)=\delta_{ij}.
\]
Geometrically, \(e^i\) is represented by the dual solid-torus state
corresponding to \(e_i\). We use the algebraic normalization above throughout.
With respect to this basis, the modular transformation
\[
s=
\begin{pmatrix}
0&-1\\
1&0
\end{pmatrix}
\]
acts by
\begin{equation}
\label{eq:torus-s-action-v2}
s_\#(e_j)=\sum_{i\in I}\frac{S_{ij}}{\D}e_i.
\end{equation}
Equivalently, let \(T,T'\) be parametrized boundary tori with coordinates
\[
p:T^2_{\mathrm{std}}\longrightarrow T,
\qquad
q:T^2_{\mathrm{std}}\longrightarrow T',
\]
and set \(q^-:=q\circ\rho:T^2_{\mathrm{std}}\to -T'\). For
\[
g=q^-\circ s\circ p^{-1}:T\longrightarrow -T',
\]
the induced map
\[
g_\#:Z_\C(T)\longrightarrow Z_\C(-T')
\]
sends the standard basis to the dual basis by the same matrix, with the label
dualized by the chosen opposite-torus parametrization:
\begin{equation}
\label{eq:edge-gluing-action-v2}
g_\#(e_j)=\sum_{i\in I}\frac{S_{ij}}{\D}e^{i^*}.
\end{equation}
See \cite[Chapter~IV, Section~5]{TuraevBook}.

\medskip
\noindent\emph{Vertex states.}

For a boundary torus \(T_h\) of a block \(P_v\), the parametrization
\[
p_h:T^2_{\mathrm{std}}\longrightarrow T_h
\]
transports the basis of \(\mathcal H_T\) to \(Z_\C(T_h)\). After this
identification, we regard \(Z_\C(T_h)\) as \(\mathcal H_T\). Recall that
\[
P_v=\Sigma_v\times S^1
\]
and
\[
\partial P_v=\bigsqcup_{h\in H(v)}T_h.
\]
By monoidality and the chosen torus parametrizations,
\[
Z_\C(\partial P_v)
\cong
\bigotimes_{h\in H(v)}Z_\C(T_h)
\cong
\bigotimes_{h\in H(v)}\mathcal H_T.
\]
Since \(P_v\) is a cobordism
\[
P_v:\varnothing\longrightarrow \bigsqcup_{h\in H(v)}T_h,
\]
it defines a vector
\[
Z_\C(P_v)\in \bigotimes_{h\in H(v)}\mathcal H_T.
\]
For a labeling \((a_h)_{h\in H(v)}\in I^{H(v)}\), define
\[
A_{\deg(v)}((a_h)_{h\in H(v)})
\]
by the expansion
\begin{equation}
\label{eq:vertex-state-expansion-v2}
Z_\C(P_v)
=
\sum_{(a_h)_{h\in H(v)}\in I^{H(v)}}
A_{\deg(v)}((a_h)_{h\in H(v)})
\bigotimes_{h\in H(v)} e_{a_h}.
\end{equation}

\medskip
\noindent\emph{The vertex coefficient.}

For \(r\geq1\), let \(A_r(a_1,\ldots,a_r)\) denote the coefficient defined by
\eqref{eq:vertex-state-expansion-v2} for the model block
\[
\Sigma_{1,r}\times S^1,
\qquad
\partial(\Sigma_{1,r}\times S^1)=\bigsqcup_{j=1}^r T_j,
\]
with boundary labels \(a_1,\ldots,a_r\).

\begin{proposition}
\label{prop:vertex-coefficient-v2}
For \(r\geq1\) and \(a_1,\ldots,a_r\in I\),
\[
A_r(a_1,\ldots,a_r)
=
\sum_{i\in I}d_i^{-r}\prod_{j=1}^r S_{i a_j}.
\]
\end{proposition}

\begin{proof}
By definition, \(A_r(a_1,\ldots,a_r)\) is obtained by pairing the boundary state
of \(\Sigma_{1,r}\times S^1\) with the tensor
\[
e^{a_1}\otimes\cdots\otimes e^{a_r}
\in
\bigotimes_{j=1}^r Z_\C(-T_j).
\]
Indeed, compatibility of \(d_\Sigma\) with disjoint union gives
\[
d_{\partial(\Sigma_{1,r}\times S^1)}
\left(
(e_{b_1}\otimes\cdots\otimes e_{b_r})
\otimes
(e^{a_1}\otimes\cdots\otimes e^{a_r})
\right)
=
\prod_{j=1}^r\delta_{a_j,b_j}.
\]
Topologically, the pairing fills each boundary torus \(T_j=C_j\times S^1\) by
the dual solid torus representing \(e^{a_j}\). The underlying closed
3-manifold after all fillings is
\[
\left(\Sigma_{1,r}\cup_{\bigsqcup_j C_j}\bigsqcup_{j=1}^r D_j^2\right)\times S^1
\cong
\overline\Sigma_1\times S^1,
\]
where \(\overline\Sigma_1\) is the closed genus-one surface obtained by capping
the boundary components of \(\Sigma_{1,r}\). The colored ribbon in the \(j\)-th
filling solid torus may be represented as the annulus
\(\alpha_j\times S^1\), where \(\alpha_j\subset D_j^2\) is an arc through the
center of the capping disk. The arc is oriented so that the boundary item
induced by the dual solid-torus state is \((V_{a_j},+1)\). After the capping
disk is attached, \(\alpha_j\) becomes a marked arc on \(\overline\Sigma_1\)
with this item. Hence the filled cobordism is the identity mapping torus of the
decorated genus-one surface with marked arcs colored by
\(V_{a_1},\ldots,V_{a_r}\).

Let
\[
Z_\C(\overline\Sigma_1;V_{a_1},\ldots,V_{a_r})
\]
denote the state space of that decorated closed genus-one surface. The trace
formula for products with \(S^1\) gives
\[
A_r(a_1,\ldots,a_r)
=
\dim_k Z_\C(\overline\Sigma_1;V_{a_1},\ldots,V_{a_r}).
\]
The genus-one Verlinde dimension formula gives
\begin{equation}
\label{eq:genus-one-verlinde-dimension-v2}
\dim_k Z_\C(\overline\Sigma_1;V_{a_1},\ldots,V_{a_r})
=
\sum_{i\in I}d_i^{-r}\prod_{j=1}^r S_{a_j i}.
\end{equation}
See \cite[Chapter~III, Theorem~2.1.3 and Chapter~IV,
Theorem~12.1.1]{TuraevBook}.
\end{proof}

For a vertex \(v\) we write \(A_{\deg(v)}((a_h)_{h\in H(v)})\) for the
corresponding function of the incident labels.

\medskip
\noindent\emph{Edge contractions.}

Let \(e=\{u,v\}\) be an edge of \(G\), and write
\[
h=h_{u,e},
\qquad
\bar h=h_{v,e}.
\]
For \(g_e\) as in \eqref{eq:graph-edge-gluing-map-v2}, if the two half-edges
\(h,\bar h\) are labeled by \(a,b\in I\), then
\eqref{eq:edge-gluing-action-v2} gives
\[
(g_e)_\#(e_a)
=
\sum_{c\in I}\frac{S_{ca}}{\D}e^{c^*}
\in Z_\C(-T_{\bar h}).
\]
Pairing with \(e_b\in Z_\C(T_{\bar h})\) gives the edge contraction
\begin{equation}
\label{eq:edge-contraction-v2}
d_{T_{\bar h}}\bigl(e_b\otimes (g_e)_\#e_a\bigr)
=
\frac{S_{a,b^*}}{\D}.
\end{equation}

\subsection{The half-edge state sum}
\label{subsec:half-edge-state-sum-v2}

The preceding local computations determine the invariant before the half-edge
labels are summed out. A half-edge labeling records the basis vector chosen on
each boundary torus of the disjoint union of vertex blocks; the edge gluings then
pair the corresponding two torus factors by \eqref{eq:edge-contraction-v2}.

\begin{proposition}
\label{prop:half-edge-state-sum-v2}
Let \(G\) be a finite connected simple graph with at least one edge. Then
\begin{equation}
\label{eq:half-edge-state-sum-v2}
\begin{aligned}
Z_\C(M_G)
&=
\D^{-|E(G)|}
\sum_{\ell:H(G)\to I}
\left(
\prod_{v\in V(G)}
A_{\deg(v)}((\ell(h))_{h\in H(v)})
\right)\\
&\qquad\qquad\cdot
\left(
\prod_{e=\{u,v\}\in E(G)}
S_{\ell(h_{u,e}),\ell(h_{v,e})^*}
\right),
\end{aligned}
\end{equation}
\end{proposition}

\begin{proof}
Let
\[
P_G^\circ=\bigsqcup_{v\in V(G)}P_v
\]
be the disjoint union of the vertex blocks. By construction, \(M_G\) is obtained
from \(P_G^\circ\) by gluing, for every edge \(e=\{u,v\}\), the torus
\[
T_{h_{u,e}}
\]
to
\[
-T_{h_{v,e}}
\]
using the maps \(g_e\) defined in \eqref{eq:graph-edge-gluing-map-v2}. The
boundary state of \(P_G^\circ\) is the tensor
product of the vertex states:
\[
Z_\C(P_G^\circ)=\bigotimes_{v\in V(G)}Z_\C(P_v)
\in
\bigotimes_{h\in H(G)}\mathcal H_T,
\]
after the identifications fixed above.

A half-edge labeling is a function \(\ell:H(G)\to I\). It assigns a
simple label to each boundary torus of the disjoint union \(P_G^\circ\), before
any edge gluing is performed. Taking the tensor product of
\eqref{eq:vertex-state-expansion-v2} over all vertices gives
\[
Z_\C(P_G^\circ)
=
\sum_{\ell:H(G)\to I}
\left(
\prod_{v\in V(G)}
A_{\deg(v)}((\ell(h))_{h\in H(v)})
\right)
\bigotimes_{h\in H(G)}e_{\ell(h)}.
\]
The gluing formula \eqref{eq:tqft-gluing-pairing-v2} is then applied to each
edge torus. Fix an edge \(e=\{u,v\}\), and write
\[
h=h_{u,e},
\qquad
\bar h=h_{v,e}.
\]
By \eqref{eq:edge-contraction-v2}, with \(a=\ell(h)\) and
\(b=\ell(\bar h)\),
\[
d_{T_{\bar h}}\bigl(e_{\ell(\bar h)}\otimes (g_e)_\#e_{\ell(h)}\bigr)
=
\frac{S_{\ell(h),\ell(\bar h)^*}}{\D},
\]
so every edge contributes one factor
\(S_{\ell(h_{u,e}),\ell(h_{v,e})^*}/\D\). Multiplying these
contributions over all edges gives the factor \(\D^{-|E(G)|}\) and the product
of the corresponding \(S\)-matrix entries.
\end{proof}

\subsection{Eliminating the half-edge labels}
\label{subsec:graph-manifold-state-sum-v2}

The half-edge formula still contains labels on boundary tori which depend on the
chosen half-edge notation. The final form of the invariant is obtained by summing
these labels edge by edge, using the square of the \(S\)-matrix.

\begin{theorem}
\label{thm:graph-manifold-rt-formula-v2}
Let \(G\) be a finite connected simple graph with at least one edge. Then
\[
Z_\C(M_G)
=
\D^{|E(G)|}
\sum_{\sigma:V(G)\to I}
\left(
\prod_{v\in V(G)}
d_{\sigma(v)}^{-\deg(v)}
\right)
\left(
\prod_{e=\{u,v\}\in E(G)}
S_{\sigma(u),\sigma(v)}
\right).
\]
\end{theorem}

\begin{proof}
We use the modular-data identity
\begin{equation}
\label{eq:square-S-for-elimination-v2}
S^2=\D^2\mathsf C,
\qquad
\mathsf C_{ij}=\delta_{i,j^*},
\end{equation}
or equivalently
\[
\sum_{a\in I}S_{ia}S_{aj}=\D^2\delta_{i,j^*}.
\]
Substituting Proposition~\ref{prop:vertex-coefficient-v2} into
Proposition~\ref{prop:half-edge-state-sum-v2} gives
\[
\begin{aligned}
Z_\C(M_G)
&=
\D^{-|E(G)|}
\sum_{\ell:H(G)\to I}
\prod_{v\in V(G)}
\left(
\sum_{i_v\in I}
d_{i_v}^{-\deg(v)}
\prod_{h\in H(v)}S_{i_v,\ell(h)}
\right)\\
&\qquad\qquad\cdot
\prod_{e=\{u,v\}\in E(G)}
S_{\ell(h_{u,e}),\ell(h_{v,e})^*}.
\end{aligned}
\]
Expanding the product of vertex sums introduces one label at each vertex; we
index these choices by a vertex labeling
\[
\sigma:V(G)\longrightarrow I.
\]
Reindexing the resulting finite sum first by \(\sigma\) and then by the
half-edge labeling \(\ell\) gives
\[
\begin{aligned}
Z_\C(M_G)
&=
\D^{-|E(G)|}
\sum_{\sigma:V(G)\to I}
\left(
\prod_{v\in V(G)}
d_{\sigma(v)}^{-\deg(v)}
\right)\\
&\qquad\cdot
\sum_{\ell:H(G)\to I}
\left(
\prod_{v\in V(G)}
\prod_{h\in H(v)}
S_{\sigma(v),\ell(h)}
\right)
\left(
\prod_{e=\{u,v\}\in E(G)}
S_{\ell(h_{u,e}),\ell(h_{v,e})^*}
\right).
\end{aligned}
\]
For fixed \(\sigma\), the remaining sum over half-edge labelings factors over
the edges. Indeed, each value \(\ell(h)\) appears once in a vertex factor and
once in the edge factor corresponding to the unique edge containing \(h\). For
an edge \(e=\{u,v\}\), the contribution of the two values of \(\ell\) on its
half-edges, denoted locally by \(a\) and \(b\), is the following two-variable
sum:
\begin{equation}
\label{eq:single-edge-elimination-v2}
\begin{aligned}
\sum_{a,b\in I}
S_{\sigma(u),a}\,S_{a,b^*}\,S_{\sigma(v),b}
&=
\sum_{b\in I}
\left(
\sum_{a\in I}S_{\sigma(u),a}S_{a,b^*}
\right)
S_{\sigma(v),b}\\
&=
\sum_{b\in I}
(S^2)_{\sigma(u),b^*}S_{\sigma(v),b}\\
&=
\sum_{b\in I}
\D^2\delta_{\sigma(u),b}S_{\sigma(v),b}\\
&=
\D^2 S_{\sigma(v),\sigma(u)}\\
&=
\D^2 S_{\sigma(u),\sigma(v)}.
\end{aligned}
\end{equation}
Applying \eqref{eq:single-edge-elimination-v2} independently to every edge,
the half-edge sum becomes
\[
\prod_{e=\{u,v\}\in E(G)}
\D^2 S_{\sigma(u),\sigma(v)}.
\]
Together with the prefactor \(\D^{-|E(G)|}\), this gives the power
\[
\D^{2|E(G)|}\D^{-|E(G)|}=\D^{|E(G)|},
\]
which proves the displayed formula.
\end{proof}

\section{The anomaly-free case}
\label{sec:anomaly-free-case-v2}

We combine the graph-manifold formula with the complexity dichotomy for
weighted graph homomorphism partition functions. After the vertex factors in
Theorem~\ref{thm:graph-manifold-rt-formula-v2} are normalized, the remaining
edge weights form a fixed matrix depending only on the modular category.

Let \(I=\Irr(\C)\). Define the symmetric matrix
\[
A_\C=(A_\C(i,j))_{i,j\in I},
\qquad
A_\C(i,j)=\frac{S_{ij}}{d_i d_j}.
\]
By the algebraicity discussion in Subsection~\ref{subsec:modular-tensor-categories-v2}, after
choosing the number field containing the relevant scalars, \(A_\C\) is a matrix
over that field.
Moreover
\[
A_\C(0,j)=A_\C(i,0)=1
\]
for all \(i,j\in I\).

Assume that \(\C\) is anomaly-free, and write
\(\D_\C=\Delta_+(\C)=\Delta_-(\C)\). The graph-manifold formula can be
rewritten as a graph partition function. For an edge \(e=\{u,v\}\),
\[
S_{\sigma(u),\sigma(v)}
=
d_{\sigma(u)}d_{\sigma(v)}
A_\C(\sigma(u),\sigma(v)).
\]
Hence
\[
\prod_{e=\{u,v\}\in E(G)}
S_{\sigma(u),\sigma(v)}
=
\left(
\prod_{v\in V(G)}
d_{\sigma(v)}^{\deg(v)}
\right)
\left(
\prod_{e=\{u,v\}\in E(G)}
A_\C(\sigma(u),\sigma(v))
\right).
\]
The first factor cancels the vertex factor in
Theorem~\ref{thm:graph-manifold-rt-formula-v2}. Therefore, for every finite
connected simple graph \(G\) with at least one edge,
\begin{equation}
\label{eq:rt-graph-partition-identity-v2}
Z_\C(M_G)=\D_\C^{|E(G)|} Z_{A_\C}(G).
\end{equation}
Here \(Z_{A_\C}(G)\) is the graph partition function of
\eqref{eq:graph-hom-partition-function-v2}. Hence RT evaluation on the
graph-manifold family determines \(Z_{A_\C}(G)\), up to the known nonzero factor
\(\D_\C^{|E(G)|}\).

We apply the graph-homomorphism dichotomy to \(A_\C\). The required algebraic
input is the following characterization.

\begin{theorem}
\label{thm:mbr1-iff-pointed-v2}
Let \(\C\) be a modular category over \(k\). Then \(A_\C\) is
multiplicative-block-rank-one if and only if \(\C\) is pointed.
\end{theorem}

\begin{proof}
Assume first that \(\C\) is pointed. Then \(I=\Irr(\C)\) is a finite abelian
group under tensor product. Let
\[
\chi_\C(a)=d_a
\]
be the spherical character. Let \(b_\C(a,b)\) be the scalar of the double
braiding on \(V_a\otimes V_b\); hence \(b_\C(a,b)=c_\C(a,b)c_\C(b,a)\) in terms
of braiding scalars. Equivalently, if \(q_\C(a)=c_\C(a,a)\), then
\[
b_\C(a,b)=\frac{q_\C(a+b)}{q_\C(a)q_\C(b)}.
\]
The ribbon twist is \(\theta_a=q_\C(a)\chi_\C(a)\), and hence also
\(\theta_{a+b}/(\theta_a\theta_b)=b_\C(a,b)\). In this pointed case, the
\(S\)-matrix entries satisfy
\[
S_{a,b}=d_{a+b}b_\C(a,b).
\]
Hence
\[
A_\C(a,b)
=
\frac{S_{a,b}}{d_ad_b}
=
\frac{d_{a+b}}{d_ad_b}b_\C(a,b)
=
b_\C(a,b),
\]
because \(d_a=\chi_\C(a)\) is a character with values in \(\{\pm1\}\). See
\cite[Section~8.4]{EGNO2015} and \cite{JoyalStreet1993} for the metric-group
description of pointed braided categories, that is, their description by finite
abelian groups equipped with non-degenerate quadratic forms.
The bicharacter \(b_\C\) has finite domain, so its values have uniformly bounded
finite order. Hence there is an integer \(N\geq1\) such that
\[
A_\C(a,b)^N=1
\qquad
\text{for all }a,b\in I.
\]
Hence \(A_\C^{\circ N}\) is the all-one matrix. In particular it is
block-rank-one, and \(A_\C\) is multiplicative-block-rank-one.

Conversely, assume that \(A_\C\) is multiplicative-block-rank-one. Choose
\(r\geq1\) such that
\[
B:=A_\C^{\circ r}
\]
is block-rank-one. Since the row and column of \(A_\C\) indexed by the tensor
unit are identically equal to \(1\), the same is true for \(B\):
\[
B_{0j}=B_{i0}=1
\qquad
\text{for all }i,j\in I.
\]
Since \(B\) is block-rank-one, its support is a disjoint union of complete
rectangular blocks in the block-diagonal sense recalled in
Section~\ref{subsec:cai-govorov-v2}. The row indexed by \(0\) has full support,
and hence every column lies in the column block paired with the row block
containing \(0\). The column indexed by \(0\) has full support, and hence every
row lies in the corresponding row block. Therefore \(B\) has a single nonzero
support block, namely all of \(I\times I\).

The single block has rank one. Applying the vanishing of the \(2\times2\) minor
with rows \(0,i\) and columns \(0,j\) gives
\[
0
=
\det
\begin{pmatrix}
B_{00}&B_{0j}\\
B_{i0}&B_{ij}
\end{pmatrix}
=
\det
\begin{pmatrix}
1&1\\
1&B_{ij}
\end{pmatrix}
=
B_{ij}-1.
\]
It follows that \(B_{ij}=1\) for all \(i,j\), or equivalently
\[
A_\C(i,j)^r=1
\qquad
\text{for all }i,j\in I.
\]
Every entry of \(A_\C\) is therefore a root of unity.

By Subsection~\ref{subsec:modular-tensor-categories-v2}, the scalars appearing in \(A_\C\) lie
in a finite extension \(K\subset k\) of \(\mathbb Q\). Embed \(K\) into
\(\mathbb C\). Under this embedding, the unitarity theorem for modular data
recalled there gives \(d_j\in\mathbb R\), \(\Dim(\C)>0\), and
\[
\sum_{j\in I}|S_{ij}|^2=\Dim(\C)
\qquad
\text{for every }i\in I.
\]
Since
\[
S_{ij}=d_i d_j A_\C(i,j),
\]
we obtain, for fixed \(i\),
\[
\begin{aligned}
\Dim(\C)
&=
\sum_{j\in I}|S_{ij}|^2\\
&=
\sum_{j\in I}|d_i d_j A_\C(i,j)|^2\\
&=
d_i^2\sum_{j\in I}d_j^2\\
&=
d_i^2\Dim(\C).
\end{aligned}
\]
Here \(|A_\C(i,j)|=1\) because \(A_\C(i,j)\) is a root of unity, and the
dimensions \(d_j\) are real. Since \(\Dim(\C)>0\), it follows that
\[
d_i^2=1
\qquad
\text{for all }i\in I.
\]

It remains to pass from categorical dimension to invertibility. By the Verlinde
diagonalization of the fusion matrices
\cite[Sections~8.13--8.14]{EGNO2015}, the eigenvalues of \(N_i\) are
\[
\frac{S_{ij}}{d_j},
\qquad j\in I.
\]
These eigenvalues can be written as
\[
\frac{S_{ij}}{d_j}
=
d_i A_\C(i,j).
\]
Each has complex absolute value \(1\), because \(d_i^2=1\) and
\(A_\C(i,j)\) is a root of unity. Hence the spectral radius of \(N_i\), which is
\(\FPdim(V_i)\), is \(1\). Therefore
\[
\FPdim(V_i)=1
\qquad
\text{for all }i\in I.
\]
A simple object in a fusion category has Frobenius--Perron dimension \(1\) if
and only if it is invertible; see
\cite[Section~4.5, Exercise~4.5.9]{EGNO2015}. Therefore every simple object of
\(\C\) is invertible, and \(\C\) is pointed.
\end{proof}

Consequently, if \(\C\) is non-pointed, then \(A_\C\) is not
multiplicative-block-rank-one. By Theorem~\ref{thm:cai-govorov-v2}, there is a
constant \(\Delta=\Delta(A_\C)\) such that
\[
\EVAL^{(\Delta)}_{\ConnSimp}(A_\C)
\]
is \(\#\mathrm P\)-hard.

For RT evaluation, the anomaly-free assumption enters through
\eqref{eq:rt-graph-partition-identity-v2}, with
\(\D_\C=\Delta_+(\C)=\Delta_-(\C)\).

\begin{theorem}
\label{thm:core-anomaly-free-hardness-v2}
Let \(\B\) be an anomaly-free modular category over \(k\). If \(\B\) is
not pointed, then
\[
\RTEVAL(\B)
\]
is \(\#\mathrm P\)-hard under polynomial-time Turing reductions.
\end{theorem}

\begin{proof}
Since \(\B\) is not pointed, Theorem~\ref{thm:mbr1-iff-pointed-v2} implies that
\(A_\B\) is not multiplicative-block-rank-one. By
Theorem~\ref{thm:cai-govorov-v2}, there is a constant
\(\Delta=\Delta(A_\B)\) such that
\[
\EVAL^{(\Delta)}_{\ConnSimp}(A_\B)
\]
is \(\#\mathrm P\)-hard.

Given a connected simple graph \(G\) with \(\Delta(G)\leq\Delta\), first consider
the case \(E(G)=\emptyset\). Then \(G\) has one vertex and
\[
Z_{A_\B}(G)=|\Irr(\B)|,
\]
a fixed integer depending only on \(\B\). For this input the reduction outputs
the fixed value and makes no oracle query. Thus assume from now on that \(G\)
has at least one edge.

By Proposition~\ref{prop:effective-surgery-v2}, one constructs in polynomial
time a planar framed-link diagram \(D_G\) representing a framed link
\(L_G\subset S^3\) whose surgery manifold is \(M_G\). The \(\RTEVAL(\B)\)
oracle returns
\[
Z_\B(M_G).
\]
Equation~\eqref{eq:rt-graph-partition-identity-v2}, applied to \(\B\), gives
\[
Z_{A_\B}(G)=\D_\B^{-|E(G)|}Z_\B(M_G).
\]
The scalar \(\D_\B\) is fixed with the category \(\B\), and is nonzero. Hence
\(\D_\B^{-|E(G)|}\) is computable in
polynomial time from \(G\) in the chosen algebraic-number model. It
follows that an oracle for
\(\RTEVAL(\B)\) computes \(Z_{A_\B}(G)\) for every connected simple graph \(G\)
of maximum degree at most \(\Delta\). This yields the polynomial-time Turing
reduction
\[
\EVAL^{(\Delta)}_{\ConnSimp}(A_\B)
\leq_T^p
\RTEVAL(\B).
\]
Since the source problem is \(\#\mathrm P\)-hard, \(\RTEVAL(\B)\) is
\(\#\mathrm P\)-hard.
\end{proof}

\begin{remark}
The proof gives the following restricted-query version. For the constant
\(\Delta=\Delta(A_\B)\), every oracle query in the reduction is one of the
framed-link presentations, produced by
Proposition~\ref{prop:effective-surgery-v2}, of a manifold in the
graph-manifold family
\[
\M_\Delta=\{M_G:\ G \text{ connected simple, } E(G)\neq\emptyset,
\text{ and } \Delta(G)\leq \Delta\}.
\]
\end{remark}

\section{Pointed modular categories}
\label{sec:pointed-modular-categories-v2}

The polynomial-time case of the RT dichotomy is independent of the
graph-manifold construction. For pointed modular categories, the RT surgery sum
is a finite abelian Gauss sum, as in \cite[Section~1.1]{Deloup1999}; see also
\cite{Deloup2001}.

Let \(\C\) be a pointed modular category. The tensor product of simple objects
makes
\[
\Lambda=\Irr(\C)
\]
a finite abelian group. The unit is \(0\), and the dual of \(a\) is \(-a\).
The spherical structure determines a character
\[
\chi_\C:\Lambda\longrightarrow \{\pm1\},
\qquad
\chi_\C(a)=d_a.
\]
Indeed, dimensions are multiplicative on invertible objects and \(d_a=d_{-a}\),
hence \(d_a^2=1\).

Choose representatives and tensor-product identifications for the simple
objects. If \(c_\C(a,b)\) denotes the scalar by which the braiding
\(V_a\otimes V_b\to V_b\otimes V_a\) acts, set
\[
q_\C\colon \Lambda\longrightarrow k^\times,
\qquad
q_\C(a)=c_\C(a,a).
\]
The polarizing bicharacter is
\[
b_\C(a,b)=\frac{q_\C(a+b)}{q_\C(a)q_\C(b)}.
\]
Equivalently, it is the scalar of the double braiding on \(V_a\otimes V_b\):
\[
b_\C(a,b)=c_\C(a,b)c_\C(b,a).
\]
The ribbon twist satisfies
\[
\theta_a=q_\C(a)\chi_\C(a).
\]
Since \(\chi_\C\) is a character, \(\theta\) is again a quadratic form on
\(\Lambda\), with the same polarizing bicharacter:
\[
\frac{\theta_{a+b}}{\theta_a\theta_b}=b_\C(a,b).
\]

Recall that a \defnfont{pre-metric group} over \(k\) is a finite abelian group
\(A\) equipped with a quadratic form \(Q:A\to k^\times\). Thus \(Q(-a)=Q(a)\)
and
\[
b_Q(a,b)=\frac{Q(a+b)}{Q(a)Q(b)}
\]
is a bicharacter. It is a \defnfont{metric group} if \(b_Q\) is non-degenerate.
With this terminology, \((\Lambda,q_\C)\) and \((\Lambda,\theta)\) are metric
groups with the same bicharacter \(b_\C\); see \cite[Section~8.4]{EGNO2015} and
\cite{JoyalStreet1993}.

We apply this to the metric group \((\Lambda,\theta)\). If
\(B=(B_{rs})\) is an integral symmetric \(m\times m\) matrix, let
\[
\theta_B:\Lambda^m\longrightarrow k^\times
\]
be the quadratic form
\[
\theta_B(x_1,\ldots,x_m)
=
\prod_{r=1}^m \theta_{x_r}^{B_{rr}}
\prod_{1\leq r<s\leq m} b_\C(x_r,x_s)^{B_{rs}}.
\]
Its unnormalized Gauss sum is
\begin{equation}
\label{eq:pointed-abelian-gauss-sum-v2}
\mathcal G_\C(B)
:=
\sum_{x\in\Lambda^m}\theta_B(x).
\end{equation}
We keep Gauss sums unnormalized; the RT normalization is written separately.

For a pointed category, the general RT constants of
Section~\ref{sec:modular-categories-quantum-invariants} take the form
\[
D_\C^2=\Dim(\C)=|\Lambda|,
\qquad
\Delta_+=\sum_{a\in\Lambda}\theta_a,
\qquad
\Delta_-=\sum_{a\in\Lambda}\theta_a^{-1}.
\]
Thus \(\Delta_+\) and \(\Delta_-\) are the one-variable Gauss sums for
\(\theta\) and \(\theta^{-1}\); in particular,
\(\Delta_+=\mathcal G_\C((1))\). The square root \(D_\C\) is the fixed rank
chosen in the RT normalization.

\begin{proposition}
\label{prop:pointed-abelian-data-v2}
Let \(\C\) be a pointed modular category over \(k\), and let \(M_L\) be obtained
by surgery on an \(m\)-component framed link \(L\) with linking matrix \(B_L\).
Then
\[
Z_\C(M_L)
=
D_\C^{-b_0(L)-1}
\Delta_+^{-b_+(L)}
\Delta_-^{-b_-(L)}
\mathcal G_\C(B_L).
\]
\end{proposition}

\begin{proof}
We specialize the surgery formula
\eqref{eq:rt-surgery-normalization-v2} to a pointed category. For simple labels
\[
x=(x_1,\ldots,x_m)\in \Lambda^m,
\]
let \(L(x)\) be the link \(L\) with its \(r\)-th component colored by
\(V_{x_r}\).
The Kirby color specializes to
\[
\Omega_\C=\sum_{a\in \Lambda}\chi_\C(a)V_a.
\]
Expanding all Kirby colors gives
\[
\langle L(\Omega_\C)\rangle_\C
=
\sum_{x\in \Lambda^m}
\left(\prod_{r=1}^m \chi_\C(x_r)\right)
\langle L(x)\rangle_\C.
\]

The ribbon-calculus evaluation of \(L(x)\) is as follows; see
\cite[Chapter~II]{TuraevBook}. Since the colors are invertible, each crossing
between two fixed colors acts by the corresponding braiding scalar. A framed
component colored by \(x_r\) contributes the dimension \(\chi_\C(x_r)\) and the
twist factor
\[
\theta_{x_r}^{(B_L)_{rr}}
\]
from its framing. A positive linking between the \(r\)-th and \(s\)-th components
contributes one double braiding, hence the scalar
\(b_\C(x_r,x_s)\). Therefore the total contribution of their linking number is
\[
b_\C(x_r,x_s)^{(B_L)_{rs}}.
\]
Hence
\[
\langle L(x)\rangle_\C
=
\left(\prod_{r=1}^m \chi_\C(x_r)\right)
\prod_{r=1}^m \theta_{x_r}^{(B_L)_{rr}}
\prod_{1\leq r<s\leq m} b_\C(x_r,x_s)^{(B_L)_{rs}}.
\]
The product of \(\chi_\C\)-factors in this display comes from the dimensions of
the colored components, while the product in the expansion of
\(\Omega_\C\) comes from the Kirby-color coefficients. Together they give
\(\prod_r \chi_\C(x_r)^2=1\). Therefore
\[
\langle L(\Omega_\C)\rangle_\C=\mathcal G_\C(B_L).
\]
Substituting this equality into
\eqref{eq:rt-surgery-normalization-v2} gives the displayed formula for
\(Z_\C(M_L)\).
\end{proof}

\begin{theorem}
\label{thm:pointed-rt-in-fp-v2}
If \(\C\) is a pointed modular category over \(k\), then
\[
\RTEVAL(\C)\in \FP.
\]
\end{theorem}

\begin{proof}
Given a framed surgery diagram \(L\), compute its linking matrix \(B_L\). The
signature and nullity of \(B_L\), and hence the normalization in
\eqref{eq:rt-surgery-normalization-v2}, are computable by integer linear
algebra in polynomial time.

It remains to evaluate \(\mathcal G_\C(B_L)\). Fix a decomposition of
\(\Lambda\) as a product of cyclic groups. Choose a primitive root of unity
\(\xi\) whose cyclic group contains all values of \(\theta\) and \(b_\C\). Then
each summand is a power of \(\xi\),
with exponent an integral quadratic form in \(O(m)\) variables and
coefficients computed from \(B_L\) modulo the exponent of \(\Lambda\).
Thus \(\theta_{B_L}\) is a quadratic form on the finite abelian group
\(\Lambda^m\), whose primary decomposition has \(O(m)\) summands because
\(\Lambda\) is fixed.

Quadratic Gauss sums on finite abelian groups given by primary decompositions
are computable in polynomial time; see
\cite[Theorem~2.2 and Appendix~A]{DelaneyMariaSampertonTY2025}. Their theorem is
stated for normalized sums with values written additively in \(\mathbb Q/\mathbb
Z\). The cited algorithm applies after expressing the fixed roots of unity as
powers of the chosen primitive root. The normalization differs
only by the known factor \(|\Lambda|^{m/2}\), which lies in the fixed number
field. Hence \(\mathcal G_\C(B_L)\), and therefore \(Z_\C(M_L)\), is computable
in polynomial time in the size of the surgery presentation.
\end{proof}

\section{Arbitrary modular categories}
\label{sec:arbitrary-modular-v2}

The non-pointed anomaly-free case proved in
Section~\ref{sec:anomaly-free-case-v2} can be used to treat an arbitrary
modular category. The anomaly-free assumption is removed by passing from a
modular category \(\C\) to the Drinfeld center of its underlying spherical fusion
category. The center is anomaly-free, and its RT invariant is recovered from the
RT invariant of \(\C\) by multiplying by the value on the oppositely oriented
3-manifold.

Let \(\C^{\mathrm{rev}}\) denote the reverse modular category, with the same
underlying fusion category as \(\C\), inverse braiding, and inverse twist. The
Drinfeld center of a modular category factors as
\[
\Z(\C)\simeq \C\boxtimes \C^{\mathrm{rev}}
\]
as ribbon, equivalently modular, categories; see \cite{Muger2003} and
\cite[Section~8.20]{EGNO2015}.

\begin{lemma}
\label{lem:center-nonpointed-v2}
Let \(\C\) be a modular category. If \(\C\) is not pointed, then \(\Z(\C)\) is
not pointed.
\end{lemma}

\begin{proof}
Choose a simple object \(X\in\C\) which is not invertible. Under the ribbon
equivalence
\[
\Z(\C)\simeq \C\boxtimes\C^{\mathrm{rev}},
\]
the object \((X,\mathbf 1)\) is simple. If it were invertible in the Deligne
product, then \(X\) would be invertible in \(\C\). This contradicts the choice
of \(X\). Hence \(\Z(\C)\) has a non-invertible simple object and is not
pointed.
\end{proof}

If \(\A\) is spherical fusion, then \(\Z(\A)\) is modular and its Gauss sums
satisfy
\[
\Delta_+(\Z(\A))=\Delta_-(\Z(\A))=\Dim(\A).
\]
See \cite[Theorem~5.4]{TuraevVirelizier2017}.

The RT invariant is multiplicative under Deligne products of modular categories,
and reversal of the braiding corresponds to reversal of the orientation of the
3-manifold. With the compatible square-root choices fixed in
Section~\ref{subsec:rt-tv-evaluation-problems-v2}, this gives
\[
Z_{\C\boxtimes\mathcal E}(M)=Z_\C(M)Z_{\mathcal E}(M),
\qquad
Z_{\C^{\mathrm{rev}}}(M)=Z_\C(-M).
\]
The surgery formula \eqref{eq:rt-surgery-normalization-v2} reflects the same
identities: the Deligne product multiplies the ribbon-link evaluations and the
reverse category inverts the braiding and twist.
For the RT product and reverse-orientation formulas, see
\cite[Chapter~IV]{TuraevBook} and
\cite[Corollary~17.8]{TuraevVirelizier2017}.
Consequently
\begin{equation}
\label{eq:center-rt-product-v2}
Z_{\Z(\C)}(M)=Z_\C(M)Z_\C(-M)
\end{equation}
for every closed oriented 3-manifold \(M\).

\begin{proposition}
\label{prop:center-rt-reduction-v2}
For every modular category \(\C\) over \(k\), there is a polynomial-time Turing
reduction
\[
\RTEVAL(\Z(\C))\leq_T^p \RTEVAL(\C).
\]
\end{proposition}

\begin{proof}
Let \(L\subset S^3\) be a framed link presenting a closed connected oriented
3-manifold \(M\) by surgery. In polynomial time one constructs a framed link
\(L^-\) presenting \(-M\), for instance by taking the mirror
diagram and changing the signs of the framings. The operation is local on the
link diagram and preserves polynomial size.

Querying the \(\RTEVAL(\C)\) oracle on \(L\) and on \(L^-\) gives \(Z_\C(M)\)
and \(Z_\C(-M)\). Their product, computed exactly in the fixed number field, is
\[
Z_\C(M)Z_\C(-M)=Z_{\Z(\C)}(M)
\]
by \eqref{eq:center-rt-product-v2}. Hence the two oracle calls compute the value
of \(\RTEVAL(\Z(\C))\) on the input \(L\), giving the required polynomial-time
Turing reduction.
\end{proof}

\begin{theorem}
\label{thm:rt-dichotomy-v2}
Let \(\C\) be a modular category over \(k\).
\begin{enumerate}
\item If \(\C\) is pointed, then
\[
\RTEVAL(\C)\in\FP.
\]
\item If \(\C\) is not pointed, then
\[
\RTEVAL(\C)
\]
is \(\#\mathrm P\)-hard under polynomial-time Turing reductions.
\end{enumerate}
\end{theorem}

\begin{proof}
The pointed case is Theorem~\ref{thm:pointed-rt-in-fp-v2}. Assume that \(\C\)
is not pointed. By Lemma~\ref{lem:center-nonpointed-v2}, the center \(\Z(\C)\)
is not pointed. The modular category \(\C\) is ribbon, hence spherical as a
fusion category. By the center theorem recalled above, \(\Z(\C)\) is an
anomaly-free modular category. Theorem~\ref{thm:core-anomaly-free-hardness-v2}
therefore implies that
\[
\RTEVAL(\Z(\C))
\]
is \(\#\mathrm P\)-hard. By Proposition~\ref{prop:center-rt-reduction-v2},
\[
\RTEVAL(\Z(\C))\leq_T^p \RTEVAL(\C).
\]
Composing reductions, every \(\#\mathrm P\) function reduces to
\(\RTEVAL(\C)\). Hence \(\RTEVAL(\C)\) is \(\#\mathrm P\)-hard.
\end{proof}

\section{The Turaev--Viro dichotomy}
\label{sec:tv-dichotomy-v2}

This section proves the Turaev--Viro dichotomy.  The hard direction follows
from the Reshetikhin--Turaev dichotomy through the Drinfeld center: the center
\(\Z(\A)\) is anomaly-free modular
\cite[Theorem~5.4]{TuraevVirelizier2017}, and the Turaev--Viro invariant of
\(\A\) agrees with the Reshetikhin--Turaev invariant of \(\Z(\A)\)
\cite[Theorem~17.1 and Corollary~17.7(a)]{TuraevVirelizier2017}.  Moreover,
\(\Z(\A)\) is pointed precisely when \(\A\) is trivializable pointed. Thus, in
the non-trivializable case, the anomaly-free hardness theorem applies to the
modular category \(\Z(\A)\).

The polynomial-time direction requires one additional calculation.  If
\(\A\simeq\Vec^{\omega,d}_\Lambda\) is trivializable pointed and \(t\) is a
triangulation, the trivializability condition lets one rewrite the pointed
Turaev--Viro state sum as an explicit finite abelian Gauss sum over the group of
\(\Lambda\)-valued \(1\)-cocycles \(Z^1(t;\Lambda)\).  This group and the
quadratic function defining the summand are computable from \(t\) by finite
abelian linear algebra, so the invariant is computable in polynomial time.

\subsection{Centers and trivializable pointed categories}
\label{subsec:tv-trivializable-pointed-v2}

The relation between Reshetikhin--Turaev and Turaev--Viro invariants is through
the Drinfeld center. By \eqref{eq:tv-rt-center-comparison-v2}, the
Turaev--Viro invariant of a spherical fusion category \(\A\) is the
Reshetikhin--Turaev invariant of \(\Z(\A)\). Thus the relevant dichotomy is
whether the center is pointed.

A pointed fusion category is tensor equivalent to
\[
\Vec^\omega_H
\]
for a finite group \(H\) and a normalized \(3\)-cocycle
\(\omega\in Z^3(H;k^\times)\). For pointed fusion categories and their
description by finite groups and \(3\)-cocycles, see
\cite[Sections~2.6 and~8.4]{EGNO2015}.

Let \(\Lambda\) be a finite abelian group. Following
\cite[Proposition~4.1]{Breen1999}, define a homomorphism
\[
\psi_\Lambda:H^3(\Lambda;k^\times)\longrightarrow
\operatorname{Hom}(\wedge^3\Lambda,k^\times)
\]
as follows. If a class is represented by a normalized \(3\)-cocycle
\(\omega\), then
\[
\psi_\Lambda(\omega)(x_1,x_2,x_3)
=
\prod_{\sigma\in S_3}
\omega(x_{\sigma(1)},x_{\sigma(2)},x_{\sigma(3)})^{\operatorname{sgn}(\sigma)}.
\]
This alternating trilinear form depends only on the cohomology class of
\(\omega\). We say that \([\omega]\in H^3(\Lambda;k^\times)\) is
\defnfont{trivializable} if
\begin{equation}
\label{eq:trivializable-cocycle-condition-v2}
\psi_\Lambda([\omega])=1.
\end{equation}

We call a spherical fusion category \(\A\) \defnfont{trivializable pointed} if
its underlying fusion category is tensor equivalent to
\[
\A\simeq \Vec^\omega_\Lambda
\]
for a finite abelian group \(\Lambda\) and a trivializable class
\([\omega]\in H^3(\Lambda;k^\times)\). Equivalently, \(\Z(\A)\) is pointed.
Indeed, if \(\Z(\A)\) is pointed, then the forgetful
tensor functor \(\Z(\A)\to\A\) is dominant, and every simple object of \(\A\) is
a subobject of the image of a direct sum of invertible objects; hence every
simple object of \(\A\) is invertible. Moreover, since \(\Z(\A)\) is braided pointed, these
invertible simples are generated by the image of a finite abelian group. Thus
\(\A\simeq\Vec^\omega_\Lambda\) with \(\Lambda\) finite abelian, and
\cite[Theorem~3.2]{AngionoGalindo2017} identifies the condition that
\(\Z(\A)\) be pointed with trivializability of \([\omega]\).

\subsection{The pointed Turaev--Viro state sum}
\label{subsec:pointed-tv-state-sum-v2}

Let \(t\) be a triangulation of a closed oriented 3-manifold \(M\), with a
total order on its vertices. An ordered edge is written
\(\beta\gamma\), \(\beta<\gamma\), and is oriented from \(\beta\) to \(\gamma\);
all simplices are written with vertices in increasing order. We write
\[
V(t),\qquad E(t),\qquad F(t),\qquad T(t)
\]
for the vertices, edges, triangles, and tetrahedra of \(t\). If
\(\Delta=[\beta\gamma\delta\epsilon]\in T(t)\), with
\(\beta<\gamma<\delta<\epsilon\), the vertex order gives an orientation of
\(\Delta\). We compare this orientation with the given orientation of \(M\) and
define
\[
\mu(\Delta)=
\begin{cases}
1,&\text{if the ordered tetrahedron agrees with the orientation of }M,\\
-1,&\text{otherwise.}
\end{cases}
\]
Let \(\Gamma\) be a finite abelian group, written additively. A
\(\Gamma\)-coloring of \(t\) is a function
\[
c:E(t)\longrightarrow \Gamma,
\qquad
\beta\gamma\longmapsto c_{\beta\gamma},
\]
on ordered edges satisfying
\begin{equation}
\label{eq:pointed-tv-coloring-relation-v2}
c_{\beta\gamma}+c_{\gamma\delta}=c_{\beta\delta}
\end{equation}
for every ordered triangle \(\beta<\gamma<\delta\). We denote the finite
abelian group of such colorings by
\[
Z^1(t;\Gamma)\subseteq \Gamma^{E(t)}.
\]

Now let \(\Lambda\) be a finite abelian group and let
\(\omega\in Z^3(\Lambda,k^\times)\) be normalized. For
\(c\in Z^1(t;\Lambda)\), define the tetrahedral weight
\[
E_\omega(c)
=
\prod_{\Delta=[\beta\gamma\delta\epsilon]\in T(t)}
\omega(c_{\beta\gamma},c_{\gamma\delta},c_{\delta\epsilon})^{-\mu(\Delta)}.
\]
Here every tetrahedron is written with vertices in increasing order, and
\(\omega\) is evaluated on the three consecutive ordered edges
\(\beta\gamma\), \(\gamma\delta\), and \(\delta\epsilon\).

As in Subsection~\ref{subsec:rt-tv-evaluation-problems-v2}, \(|M|_\A\) denotes the
Turaev--Viro invariant of the spherical fusion category \(\A\). For the
spherical structure \(d=1\) on \(\Vec^{\omega}_\Lambda\), the proof of
\cite[Theorem~H.1]{TuraevVirelizier2017}, specifically formulas \((H.1)\),
\((H.2)\), and \((H.6)\), gives
\begin{equation}
\label{eq:pointed-tv-canonical-state-sum-v2}
|M|_{\Vec^{\omega,1}_\Lambda}
=
|\Lambda|^{-|V(t)|}
\sum_{c\in Z^1(t;\Lambda)}
E_\omega(c).
\end{equation}

For a character
\[
d:\Lambda\to\{\pm 1\},
\]
write
\[
\Vec^{\omega,d}_\Lambda
\]
for the pointed spherical category with this character. The associator is still
\(\omega\). Hence \(Z^1(t;\Lambda)\) and \(E_\omega(c)\) are unchanged. By
Appendix~A.3 and Example~2.7.2 of \cite{TuraevVirelizier2017}, the effect of
\(d\) on the summand is a product of values of \(d\) on the edge labels. This
product defines a character
\[
\chi_{d,t}:Z^1(t;\Lambda)\longrightarrow \{\pm1\}.
\]
The pointed state-sum formula becomes
\begin{equation}
\label{eq:pointed-tv-spherical-state-sum-v2}
|M|_{\Vec^{\omega,d}_\Lambda}
=
|\Lambda|^{-|V(t)|}
\sum_{c\in Z^1(t;\Lambda)}
\chi_{d,t}(c)\,E_\omega(c),
\end{equation}
with \(\chi_{d,t}=1\) when \(d=1\).

\subsection{A normal form for trivializable 3-cocycles}
\label{subsec:degree-three-normal-form-v2}

We record a normal form for \(H^3(\Lambda;k^\times)\). It separates
the part detected by \(\psi_\Lambda\) from the part which produces quadratic
weights on edge colorings.

\subsubsection{The K\"unneth decomposition}

Fix a decomposition
\[
\Lambda\simeq \bigoplus_{i=1}^r C_i,
\qquad
C_i=\mathbb Z/n_i\mathbb Z.
\]
For a cyclic group \(C_n\), one has
\[
H_{2s+1}(C_n;\mathbb Z)\simeq C_n,
\qquad
H_{2s}(C_n;\mathbb Z)=0
\]
for \(s\geq0\) and \(2s>0\). Applying the K\"unneth formula for group homology to
the direct product \(\Lambda=\bigoplus_i C_i\) gives
\[
H_3(\Lambda;\mathbb Z)\simeq
\left(\bigoplus_i \mathbb Z/n_i\mathbb Z\right)
\oplus
\left(\bigoplus_{i<j}\mathbb Z/g_{ij}\mathbb Z\right)
\oplus
\left(\bigoplus_{i<j<\ell}\mathbb Z/g_{ij\ell}\mathbb Z\right),
\]
where
\[
g_{ij}=\gcd(n_i,n_j),
\qquad
g_{ij\ell}=\gcd(n_i,n_j,n_\ell).
\]
The first summand comes from the \(H_3\) of one cyclic factor. The second comes
from the two-factor K\"unneth terms. The third is generated by products of three
degree-one classes under the Pontryagin product on group homology. We denote
\[
A_3(\Lambda)=
\bigoplus_{i<j<\ell}\mathbb Z/g_{ij\ell}\mathbb Z
\]
and let \(Q_3(\Lambda)\) be the direct sum of the one-factor and two-factor
summands. Thus
\[
H_3(\Lambda;\mathbb Z)\simeq Q_3(\Lambda)\oplus A_3(\Lambda).
\]

Since \(k^\times\) is divisible, it is injective as a \(\mathbb Z\)-module. The
universal coefficient theorem gives
\[
H^3(\Lambda;k^\times)\simeq \Hom(H_3(\Lambda;\mathbb Z),k^\times).
\]
A degree-three cohomology class is therefore a character on this homology
group. We use the following normalized bar \(3\)-cocycles.

\subsubsection{Explicit representatives}

For \(C_n=\mathbb Z/n\mathbb Z\), write \(\overline a\in\{0,\ldots,n-1\}\) for
the representative of \(a\in C_n\). Define
\[
k_n:C_n\times C_n\longrightarrow \mathbb Z,
\qquad
k_n(b,c)=
\begin{cases}
1,&\overline b+\overline c\geq n,\\
0,&\overline b+\overline c<n.
\end{cases}
\]
It satisfies
\[
k_n(c,d)-k_n(b+c,d)+k_n(b,c+d)-k_n(b,c)=0,
\]
the cocycle identity induced by associativity modulo \(n\).

Fix a primitive \(m\)-th root of unity \(\zeta_m\in k^\times\) for every
integer \(m\) appearing in these formulas. For \(a,b,c\in\Lambda\), set
\[
\begin{aligned}
\eta_i(a,b,c)
&=
\zeta_{n_i}^{\,\overline{a_i}\,k_{n_i}(b_i,c_i)}
&& (1\leq i\leq r),\\
\eta_{ij}(a,b,c)
&=
\zeta_{g_{ij}}^{\,\overline{a_i}^{(g_{ij})}\,
k_{n_j}(b_j,c_j)}
&& (i<j),\\
\Theta_{ij\ell}(a,b,c)
&=
\zeta_{g_{ij\ell}}^{
\overline{a_i}^{(g_{ij\ell})}
\overline{b_j}^{(g_{ij\ell})}
\overline{c_\ell}^{(g_{ij\ell})}}
&& (i<j<\ell).
\end{aligned}
\]
Here \(\overline{a_i}^{(g)}\in\{0,\ldots,g-1\}\) denotes the reduction of
\(a_i\) modulo \(g\), and similarly for the other coordinates.

The powers of \(\eta_i\), \(\eta_{ij}\), and \(\Theta_{ij\ell}\) represent,
respectively, the cohomology classes dual to the one-factor, two-factor, and
three-factor generators in the K\"unneth decomposition of
\(H_3(\Lambda;\mathbb Z)\). Together they give the generators of
\(H^3(\Lambda;k^\times)\) under this universal-coefficient isomorphism.
These formulas are the bar representatives coming from the cyclic resolutions;
see
\cite[Section~2.1]{GalindoMorales2018} and
\cite[Sections~2--3]{HuangLiuYe2014}.

\subsubsection{The alternating component}

For a normalized \(3\)-cocycle \(\omega\),
\[
\psi_\Lambda(\omega)(a,b,c)
=
\frac{
\omega(a,b,c)\omega(c,a,b)\omega(b,c,a)
}{
\omega(a,c,b)\omega(c,b,a)\omega(b,a,c)
}.
\]
Breen's construction gives from this expression a homomorphism depending only on
\([\omega]\):
\[
\psi_\Lambda:H^3(\Lambda;k^\times)
\longrightarrow
\Hom(\wedge^3\Lambda,k^\times)
\]
\cite[Proposition~4.1]{Breen1999}.

The cocycles \(\eta_i\) and \(\eta_{ij}\) have trivial alternating component:
\[
\psi_\Lambda(\eta_i)=1,
\qquad
\psi_\Lambda(\eta_{ij})=1.
\]
They represent the cohomology classes for characters of \(Q_3(\Lambda)\).

For \(\Theta_{ij\ell}\), if \(i<j<\ell\) and \(g=g_{ij\ell}\), then
\[
\psi_\Lambda(\Theta_{ij\ell})(a,b,c)
=
\zeta_g^{
\det
\begin{pmatrix}
\overline{a_i}^{(g)}&\overline{a_j}^{(g)}&\overline{a_\ell}^{(g)}\\
\overline{b_i}^{(g)}&\overline{b_j}^{(g)}&\overline{b_\ell}^{(g)}\\
\overline{c_i}^{(g)}&\overline{c_j}^{(g)}&\overline{c_\ell}^{(g)}
\end{pmatrix}
}.
\]
In particular, if \(e_i,e_j,e_\ell\) are generators of the three
cyclic factors, then
\[
\psi_\Lambda(\Theta_{ij\ell})(e_i,e_j,e_\ell)=\zeta_g.
\]
Thus the \(\Theta_{ij\ell}\) factors map to the generators of
\(\Hom(\wedge^3\Lambda,k^\times)\). Equivalently, under
\[
H^3(\Lambda;k^\times)\simeq \Hom(H_3(\Lambda;\mathbb Z),k^\times),
\]
the map \(\psi_\Lambda\) is restriction of a character from
\(H_3(\Lambda;\mathbb Z)\) to \(A_3(\Lambda)\). Hence
\[
\ker(\psi_\Lambda)
\simeq
\Hom(Q_3(\Lambda),k^\times).
\]
Thus \eqref{eq:trivializable-cocycle-condition-v2} implies that, after replacing
\(\omega\) by a cohomologous normalized cocycle, we may use a representative
built only from the \(\eta_i\) and \(\eta_{ij}\) factors:
\[
\eta_i,
\qquad
\eta_{ij}\quad (i<j).
\]

\subsection{Quadratic Gauss sums in the trivializable pointed case}
\label{subsec:trivializable-pointed-tv-fp-v2}

Assume that \eqref{eq:trivializable-cocycle-condition-v2} holds, and replace
\(\omega\) by the normalized representative from
Subsection~\ref{subsec:degree-three-normal-form-v2} built only from the
\(\eta_i\) and \(\eta_{ij}\) factors. We insert this representative in the
pointed state sum and obtain a quadratic form on the coloring group.

Let \(t\) be an ordered triangulation of a closed oriented 3-manifold \(M\),
and let
\[
c\in Z^1(t;\Lambda)
\]
be a \(\Lambda\)-coloring. The tetrahedral weight is
\[
E_\omega(c)
=
\prod_{\Delta=[\beta\gamma\delta\epsilon]\in T(t)}
\omega(c_{\beta\gamma},c_{\gamma\delta},c_{\delta\epsilon})^{-\mu(\Delta)},
\]
with \(\mu(\Delta)\) as in Subsection~\ref{subsec:pointed-tv-state-sum-v2}.

Write
\[
\Lambda\simeq \bigoplus_{i=1}^r C_i,
\qquad
C_i=\mathbb Z/n_i\mathbb Z
\]
and decompose
\[
c=(c_1,\ldots,c_r),
\qquad
c_i\in Z^1(t;C_i).
\]
Choose lifts \(\widetilde c_i(e)\in\{0,\ldots,n_i-1\}\) of the edge labels.

By this choice of representative, we may write
\[
\omega=
\prod_i \eta_i^{A_i}
\prod_{i<j}\eta_{ij}^{A_{ij}}
\]
for fixed integers \(A_i\), taken modulo \(n_i\), and \(A_{ij}\), taken modulo
\(g_{ij}\). On a tetrahedron \(\Delta=[\beta\gamma\delta\epsilon]\), the factors
are
\[
\eta_i(c_{\beta\gamma},c_{\gamma\delta},c_{\delta\epsilon})
=
\zeta_{n_i}^{
\widetilde c_i(\beta\gamma)
k_{n_i}\bigl(c_i(\gamma\delta),c_i(\delta\epsilon)\bigr)}
\]
and, for \(i<j\),
\[
\eta_{ij}(c_{\beta\gamma},c_{\gamma\delta},c_{\delta\epsilon})
=
\zeta_{g_{ij}}^{
\overline{c_i(\beta\gamma)}^{(g_{ij})}
k_{n_j}\bigl(c_j(\gamma\delta),c_j(\delta\epsilon)\bigr)}.
\]
Here \(\overline{c_i(\beta\gamma)}^{(g_{ij})}\) denotes reduction modulo
\(g_{ij}\). Define
\[
L_j^\Delta(\widetilde c)
=
\widetilde c_j(\gamma\delta)+\widetilde c_j(\delta\epsilon)
-\widetilde c_j(\gamma\epsilon).
\]
Equation~\eqref{eq:pointed-tv-coloring-relation-v2}, applied to the ordered
face \(\gamma<\delta<\epsilon\) and projected to \(C_j\), gives
\[
c_j(\gamma\delta)+c_j(\delta\epsilon)=c_j(\gamma\epsilon)
\quad\text{in } C_j.
\]
Since the lifts lie in \(\{0,\ldots,n_j-1\}\), this is equivalent to
\[
L_j^\Delta(\widetilde c)
=
n_j\,k_{n_j}\bigl(c_j(\gamma\delta),c_j(\delta\epsilon)\bigr).
\]

Choose an integer \(N\) divisible by \(2\), by all \(n_i^2\), and by all
\(g_{ij}n_j\), and choose a primitive \(N\)-th root of unity \(\xi\). Since all
roots of unity in the representatives \(\eta_i\) and \(\eta_{ij}\) have order
dividing \(N\),
there are fixed integers \(u_i\) and
\(u_{ij}\) such that
\[
\zeta_{n_i}=\xi^{u_iN/n_i},
\qquad
\zeta_{g_{ij}}=\xi^{u_{ij}N/g_{ij}}.
\]
Define, on the integer lifts of the edge coordinates, the integral polynomial
\[
\begin{aligned}
q_t(\widetilde c)
&=
\sum_{\Delta=[\beta\gamma\delta\epsilon]\in T(t)}
-\mu(\Delta)
\Bigg[
\sum_i
A_i u_i\,\frac{N}{n_i^2}\,
\widetilde c_i(\beta\gamma)\,L_i^\Delta(\widetilde c)
\\
&\hspace{4.5cm}
+
\sum_{i<j}
A_{ij}u_{ij}\,\frac{N}{g_{ij}n_j}\,
\widetilde c_i(\beta\gamma)\,L_j^\Delta(\widetilde c)
\Bigg].
\end{aligned}
\]

\begin{lemma}
\label{lem:tv-weight-quadratic-form-v2}
For a normalized representative of the form
\[
\omega=
\prod_i \eta_i^{A_i}
\prod_{i<j}\eta_{ij}^{A_{ij}},
\]
the function
\[
Q_\omega:Z^1(t;\Lambda)\longrightarrow k^\times,
\qquad
Q_\omega(c)=E_\omega(c),
\]
is a quadratic form.
\end{lemma}

\begin{proof}
For \(c\in Z^1(t;\Lambda)\), let
\[
\bar q_t(c)=q_t(\widetilde c)\pmod N,
\]
where the lifts \(\widetilde c_i(e)\) are chosen in \(\{0,\ldots,n_i-1\}\).
For the \(\eta_{ij}\)-terms, replacing \(\widetilde c_i(\beta\gamma)\) by its
reduction modulo \(g_{ij}\) gives the same value modulo \(N\), because
\(L_j^\Delta(\widetilde c)\) is divisible by \(n_j\). Hence the formula for
\(q_t\) is exactly the total exponent of the product of the \(\eta_i\)- and
\(\eta_{ij}\)-factors, and
\[
E_\omega(c)=\xi^{\bar q_t(c)}.
\]

Let
\[
B_t(c,c')=
\bar q_t(c+c')-\bar q_t(c)-\bar q_t(c')
\quad\text{in }\mathbb Z/N\mathbb Z.
\]
Because \(q_t\) is a sum of the terms displayed below, bilinearity of \(B_t\)
reduces to the corresponding calculation for one such term. Write
\[
m_{ii}=n_i,\qquad m_{ij}=g_{ij}\quad (i<j),
\]
and let \(\lambda_{ij}\) denote the corresponding coefficient in \(q_t\), namely
\[
\lambda_{ii}=A_i u_i\frac{N}{n_i^2},
\qquad
\lambda_{ij}=A_{ij}u_{ij}\frac{N}{g_{ij}n_j}\quad (i<j).
\]
The elementary summand has the form
\[
q_{ij}(c)=
\sum_{\Delta=[\beta\gamma\delta\epsilon]}
-\mu(\Delta)\lambda_{ij}\,
\widetilde c_i(\beta\gamma)L_j^\Delta(\widetilde c),
\]
with \(i=j\) for the \(\eta_i\)-terms and \(i<j\) for the \(\eta_{ij}\)-terms.

For \(c,c'\in Z^1(t;\Lambda)\), set
\[
s_\nu(e)=k_{n_\nu}(c_\nu(e),c'_\nu(e)).
\]
Then
\[
\widetilde{(c+c')_\nu}(e)
=
\widetilde c_\nu(e)+\widetilde c'_\nu(e)-n_\nu s_\nu(e),
\]
and, for \(\Delta=[\beta\gamma\delta\epsilon]\),
\[
L_\nu^\Delta(\widetilde{c+c'})
=
L_\nu^\Delta(\widetilde c)+L_\nu^\Delta(\widetilde c')
-n_\nu\bigl(s_\nu(\gamma\delta)+s_\nu(\delta\epsilon)-s_\nu(\gamma\epsilon)\bigr).
\]
Since \(m_{ij}\mid n_i\), the correction in the first factor contributes a multiple
of \(N\). The correction in the second factor is handled by the following
identity: if \(a\in Z^1(t;\mathbb Z/m\mathbb Z)\) and \(s\) is any
\(\mathbb Z/m\mathbb Z\)-valued edge cochain, then
\begin{equation}
\label{eq:tv-summation-by-parts-v2}
\sum_{\Delta=[\beta\gamma\delta\epsilon]}
\mu(\Delta)\,
a(\beta\gamma)
\bigl(s(\gamma\delta)+s(\delta\epsilon)-s(\gamma\epsilon)\bigr)
=0
\quad\text{in }\mathbb Z/m\mathbb Z.
\end{equation}
To verify \eqref{eq:tv-summation-by-parts-v2}, define the \(2\)-cochain
\[
u(\beta\gamma\delta)=a(\beta\gamma)s(\gamma\delta).
\]
For \(\Delta=[\beta\gamma\delta\epsilon]\), the defining relation
\eqref{eq:pointed-tv-coloring-relation-v2} for \(a\) gives
\(a(\beta\delta)=a(\beta\gamma)+a(\gamma\delta)\), and hence
\[
\delta u(\Delta)
=
-a(\beta\gamma)
\bigl(s(\gamma\delta)+s(\delta\epsilon)-s(\gamma\epsilon)\bigr).
\]
Thus the displayed sum is the negative of the sum of \(\delta u\) over the
oriented tetrahedra of \(t\). Since \(M\) is closed, each oriented triangle occurs
twice with opposite signs in this total boundary, so the sum is zero.

Applying \eqref{eq:tv-summation-by-parts-v2} with \(m=m_{ij}\) and with \(a\)
equal to the reduction of
\(c_i+c'_i\) modulo \(m_{ij}\), the second-factor correction contributes a
multiple of
\[
\lambda_{ij}n_jm_{ij}=0\quad\text{in }\mathbb Z/N\mathbb Z.
\]
Thus the polarization of \(q_{ij}\) is
\[
B_{ij}(c,c')
=
\sum_{\Delta=[\beta\gamma\delta\epsilon]}
-\mu(\Delta)\lambda_{ij}
\left(
\widetilde c_i(\beta\gamma)L_j^\Delta(\widetilde c')
+\widetilde c'_i(\beta\gamma)L_j^\Delta(\widetilde c)
\right)
\pmod N.
\]
Additivity in each argument follows from \eqref{eq:tv-summation-by-parts-v2}.
Replacing \(c\) by \(c+c''\) introduces two kinds of extra terms: the correction
in the first factor and the corresponding correction in \(L_j^\Delta\). The
first gives a multiple of \(N\), because
\(m_{ij}\mid n_i\) and \(L_j^\Delta(\widetilde c')\) is divisible by \(n_j\).
The second vanishes by \eqref{eq:tv-summation-by-parts-v2} with \(a\) equal to
the reduction of \(c'_i\) modulo \(m_{ij}\). The verification of additivity in
the second argument is identical.
Therefore
\[
B_t=\sum_i B_{ii}+\sum_{i<j}B_{ij}
\]
is bilinear.

Finally, \(B_t(c,c)=2\bar q_t(c)\). Since \(\bar q_t(0)=0\) and \(B_t\) is
bilinear,
\[
0=\bar q_t(c-c)=\bar q_t(c)+\bar q_t(-c)-B_t(c,c),
\]
so \(\bar q_t(-c)=\bar q_t(c)\). Hence \(Q_\omega(c)=\xi^{\bar q_t(c)}\) is a
quadratic form on \(Z^1(t;\Lambda)\).
\end{proof}

For each tetrahedron \(\Delta\), the formula for \(q_t\) adds a fixed number of
monomials
\[
\widetilde c_i(e)\widetilde c_j(e')
\]
with coefficients in \(\mathbb Z/N\mathbb Z\). Since \(\Lambda\), \(\omega\),
and \(N\) are fixed, the coefficient matrix of \(q_t\) is computed in time
linear in \(|T(t)|\).

Let \(d:\Lambda\to\{\pm1\}\) be a character, and let
\(\chi_{d,t}\) be the character appearing in
\eqref{eq:pointed-tv-spherical-state-sum-v2}. Define
\[
Q_{\omega,d,t}:Z^1(t;\Lambda)\longrightarrow k^\times,
\qquad
Q_{\omega,d,t}(c)=\chi_{d,t}(c)E_\omega(c).
\]
The unnormalized Gauss sum is
\[
\mathcal G_{\omega,d}(t)
:=
\sum_{c\in Z^1(t;\Lambda)}Q_{\omega,d,t}(c).
\]
With this notation, the pointed state-sum formula
\eqref{eq:pointed-tv-spherical-state-sum-v2} becomes
\begin{equation}
\label{eq:tv-pointed-gauss-sum-v2}
|M|_{\Vec^{\omega,d}_\Lambda}
=
|\Lambda|^{-|V(t)|}\mathcal G_{\omega,d}(t).
\end{equation}
By Lemma~\ref{lem:tv-weight-quadratic-form-v2}, \(E_\omega\) is a quadratic
form on \(Z^1(t;\Lambda)\). The factor \(\chi_{d,t}\) is a character, hence
contributes only a linear term to the exponent. Since \(\chi_{d,t}\) has values
in \(\{\pm1\}\), \(Q_{\omega,d,t}\) is again a computable quadratic form on
\(Z^1(t;\Lambda)\).
Thus \(\mathcal G_{\omega,d}(t)\) is the unnormalized Gauss sum of the finite
pre-metric group \((Z^1(t;\Lambda),Q_{\omega,d,t})\).

\begin{theorem}
\label{thm:trivializable-pointed-tv-in-fp-v2}
If \(\A\) is trivializable pointed, then
\[
\TVEVAL(\A)\in\FP.
\]
\end{theorem}

\begin{proof}
Write \(\A\) as a spherical pointed fusion category
\[
\A\simeq \Vec^{\omega,d}_\Lambda,
\]
where \(\Lambda\) is finite abelian, \([\omega]\) satisfies
\eqref{eq:trivializable-cocycle-condition-v2}, and
\(d:\Lambda\to\{\pm1\}\) is the character.
By \eqref{eq:trivializable-cocycle-condition-v2},
Subsection~\ref{subsec:degree-three-normal-form-v2} gives a cohomologous
normalized representative built only from the \(\eta_i\) and \(\eta_{ij}\)
factors. Cohomologous
normalized cocycles define equivalent pointed categories by the \(2\)-cochain
monoidal equivalence described in \cite[Appendix~A.3]{TuraevVirelizier2017}.
Transporting the spherical structure along this equivalence keeps the same
dimension character \(d\). Hence replacing \(\omega\) by this representative does
not change the Turaev--Viro invariant. We do this replacement and keep the
notation \(\omega\).

Let \(t\) be an ordered triangulation of a closed connected oriented
3-manifold \(M\). By \eqref{eq:tv-pointed-gauss-sum-v2},
\[
|M|_{\Vec^{\omega,d}_\Lambda}
=
|\Lambda|^{-|V(t)|}\mathcal G_{\omega,d}(t).
\]
Define the homomorphism
\[
\delta_t:\Lambda^{E(t)}\longrightarrow \Lambda^{F(t)}
\]
by
\begin{equation}
\label{eq:tv-coboundary-coordinate-v2}
(\delta_t c)_{\beta\gamma\delta}
=
c_{\gamma\delta}-c_{\beta\delta}+c_{\beta\gamma}
\end{equation}
on each ordered triangle \(\beta<\gamma<\delta\). Then
\eqref{eq:pointed-tv-coloring-relation-v2} gives
\[
Z^1(t;\Lambda)=\ker(\delta_t).
\]
The matrix of \(\delta_t\) has entries \(0,\pm1\), interpreted in the fixed
cyclic factors of \(\Lambda\). Smith and Hermite normal form compute, in
polynomial time, a decomposition of \(\ker(\delta_t)\) as a direct sum of cyclic groups, together
with generators in the original edge coordinates \cite{KannanBachem1979}.

The discussion preceding the proposition and
Lemma~\ref{lem:tv-weight-quadratic-form-v2} show that \(Q_{\omega,d,t}\) is a
computable quadratic form on this kernel. Pulling it back to the computed
decomposition gives a quadratic Gauss sum over a finite abelian group given by
cyclic factors. Such sums, including the degenerate case, are computable in
polynomial time by
\cite[Theorem~2.2 and Appendix~A]{DelaneyMariaSampertonTY2025}. The scalar
\(|\Lambda|^{-|V(t)|}\) is part of the same arithmetic model.

Thus \(\TVEVAL(\A)\in\FP\).
\end{proof}

\subsection{Proof of the dichotomy}
\label{subsec:tv-final-dichotomy-v2}

\begin{theorem}
\label{thm:tv-dichotomy-v2}
Let \(\A\) be a spherical fusion category over \(k\).
\begin{enumerate}
\item If \(\A\) is trivializable pointed, then
\[
\TVEVAL(\A)\in\FP.
\]
\item If \(\A\) is not trivializable pointed, then
\[
\TVEVAL(\A)
\]
is \(\#\mathrm P\)-hard under polynomial-time Turing reductions.
\end{enumerate}
\end{theorem}

\begin{proof}
Assume first that \(\A\) is trivializable pointed. Then
\(\TVEVAL(\A)\in\FP\) by
Theorem~\ref{thm:trivializable-pointed-tv-in-fp-v2}.

Now assume that \(\A\) is not trivializable pointed. By the criterion recalled in
Subsection~\ref{subsec:tv-trivializable-pointed-v2}, the modular category
\(\Z(\A)\) is not pointed. The center is anomaly-free modular by
\cite[Theorem~5.4]{TuraevVirelizier2017}. Hence
Theorem~\ref{thm:core-anomaly-free-hardness-v2}, applied to \(\B=\Z(\A)\),
gives a constant \(\Delta\) such that evaluation of
\[
G\longmapsto Z_{\Z(\A)}(M_G)
\]
is \(\#\mathrm P\)-hard for connected simple graphs \(G\) of maximum degree at
most \(\Delta\).

For the one-vertex edgeless graph, the corresponding graph-partition value is
fixed; on that input the reduction outputs this value and makes no oracle query.
Thus the oracle queries in the reduction may be
taken to come from connected simple graphs with at least one edge. Given such a
graph \(G\), construct in
polynomial time a triangulation of \(M_G\) using
Proposition~\ref{prop:effective-triangulation-v2}.
By the Turaev--Viro/RT comparison
\eqref{eq:tv-rt-center-comparison-v2},
\[
|M_G|_\A=Z_{\Z(\A)}(M_G).
\]
Hence an oracle for \(\TVEVAL(\A)\) evaluates
\[
G\longmapsto Z_{\Z(\A)}(M_G)
\]
on the bounded-degree graph-manifold family. Together with the polynomial-time
construction of \(t_G\), this gives a polynomial-time Turing reduction.
\end{proof}

\begin{remark}
The proof gives the following restricted-query version. There is a constant
\(\Delta\), depending only on \(\A\), such that the hardness reduction may be
chosen so that every oracle query is the triangulation, produced by
Proposition~\ref{prop:effective-triangulation-v2}, of a manifold \(M_G\), where
\(G\) is connected simple, \(E(G)\neq\emptyset\), and
\(\Delta(G)\leq\Delta\).
\end{remark}

\bibliographystyle{amsalpha}
\bibliography{references}

\end{document}